\documentclass[11pt]{amsart}

\usepackage{amssymb, amscd}
\usepackage{epsfig, mathtools, tensor}
\usepackage[vmargin=1in, hmargin=1.25in]{geometry}
\usepackage[font=small,format=plain,labelfont=bf,up,textfont=it,up]{caption}
\usepackage{microtype}
\usepackage[shortlabels]{enumitem}
\setlist[itemize]{nosep}
\usepackage[backref=page, bookmarks, bookmarksdepth=2, colorlinks=true, linkcolor=blue, citecolor=blue, urlcolor=blue]{hyperref}
\usepackage{etoolbox}
\usepackage{tikz-cd}
\usepackage{tabularray}
\apptocmd{\thebibliography}{\raggedright}{}{}
\usepackage[only,llparenthesis, rrparenthesis, llbracket,rrbracket]{stmaryrd}
\SetSymbolFont{stmry}{bold}{U}{stmry}{m}{n}
\usepackage{bm}
\usepackage{refcount}

\mathtoolsset{showonlyrefs}

\newcommand{\Figure}[1]{%
  \par
  \nointerlineskip
  \vspace{5pt}%
  \noindent
  \makebox[\linewidth][c]{%
    \psfig{file=#1,scale=1}%
  }%
  \par
  \nointerlineskip
  \vspace{5pt}%
  \noindent\ignorespaces
}

\makeatletter
\patchcmd{\@maketitle}{\global\topskip42\p@\relax}
  {\global\topskip42\p@\relax \vspace*{-38pt}}
  {}{}
\makeatother

\setenumerate[0]{label=(\alph*)}

\renewcommand*{\backref}[1]{}
\renewcommand*{\backrefalt}[4]{%
    \ifcase #1 (Not cited.)%
    \or        (Cited on page~#2.)%
    \else      (Cited on pages~#2.)%
    \fi}

\DeclareSymbolFont{bbold}{U}{bbold}{m}{n}
\DeclareSymbolFontAlphabet{\mathbbold}{bbold}

\DeclareSymbolFont{extraup}{U}{zavm}{m}{n}
\DeclareMathSymbol{\varheartsuit}{\mathalpha}{extraup}{86}

\newcommand{\arxiv}[1]{\href{http://arxiv.org/abs/#1}{{\tt arXiv:#1}}}
\numberwithin{equation}{section}

\theoremstyle{plain}
\newtheorem{theorem}{Theorem}[section]
\newtheorem{maintheorem}{Theorem}
\newtheorem{maincorollary}[maintheorem]{Corollary}
\newtheorem{proposition}[theorem]{Proposition}
\newtheorem{lemma}[theorem]{Lemma}

\newtheorem*{unnumberedclaim}{Claim}

\newtheorem{stepx}{Step}
\newenvironment{step}[1]
 {\renewcommand\thestepx{#1}\stepx}
 {\endstepx}

\providecommand{\previous}{}
\newtheorem*{primedtheoreminner}{Theorem \previous}
\makeatletter
\newenvironment{primedtheorem}[1]{%
  \edef\previous{\getrefnumber{#1}$'$}%
  \let\@currentlabel\previous
  \primedtheoreminner
}{\endprimedtheoreminner}
\makeatother

\providecommand{\previous}{}
\newtheorem*{primeprimedtheoreminner}{Theorem \previous}
\makeatletter
\newenvironment{primeprimedtheorem}[1]{%
  \edef\previous{\getrefnumber{#1}$''$}%
  \let\@currentlabel\previous
  \primeprimedtheoreminner
}{\endprimeprimedtheoreminner}
\makeatother

\providecommand{\previous}{}



\theoremstyle{definition}
\newtheorem{asm}[theorem]{Assumption}

\newtheorem*{genusasm}{Genus assumption}
\newenvironment{genusassumption}[1][]{\begin{genusasm}[#1]\pushQED{\qed}}{\popQED \end{genusasm}}
\newtheorem{defn}[theorem]{Definition}

\newtheorem{notn}[theorem]{Notation}

\newtheorem{covn}[theorem]{Convention}

\theoremstyle{remark}
\newtheorem{rmk}[theorem]{Remark}
\newenvironment{remark}[1][]{\begin{rmk}[#1] \pushQED{\qed}}{\popQED \end{rmk}}
\newtheorem{eg}[theorem]{Example}

\theoremstyle{plain}

\DeclareMathOperator{\Hom}{Hom}

\DeclareMathOperator{\Image}{Im}

\DeclareMathOperator{\Sp}{Sp}

\newcommand\Z{\ensuremath{\mathbb{Z}}}

\DeclareMathOperator{\HH}{H}

\DeclareMathOperator{\Int}{Int}

\newcommand\Span[1]{\ensuremath{\langle #1 \rangle}}

\newcommand\Set[2]{\ensuremath{\left\{\text{#1 $|$ #2}\right\}}}

\newcommand\cD{\ensuremath{\mathcal{D}}}

\newcommand\cH{\ensuremath{\mathcal{H}}}
\newcommand\cI{\ensuremath{\mathcal{I}}}

\newcommand\cK{\ensuremath{\mathcal{K}}}

\newcommand\cP{\ensuremath{\mathcal{P}}}
\newcommand\cQ{\ensuremath{\mathcal{Q}}}
\newcommand\cR{\ensuremath{\mathcal{R}}}

\newcommand\cT{\ensuremath{\mathcal{T}}}

\newcommand\fB{\ensuremath{\mathfrak{B}}}

\newcommand\fS{\ensuremath{\mathfrak{S}}}

\newcommand\bbF{\ensuremath{\mathbb{F}}}

\newcommand\bbS{\ensuremath{\mathbb{S}}}

\newcommand\oH{\ensuremath{\overline{H}}}

\newcommand\oS{\ensuremath{\overline{S}}}

\newcommand\oU{\ensuremath{\overline{U}}}

\newcommand\oa{\ensuremath{\overline{a}}}
\newcommand\ob{\ensuremath{\overline{b}}}

\newcommand\oh{\ensuremath{\overline{h}}}

\newcommand\os{\ensuremath{\overline{s}}}
\newcommand\ot{\ensuremath{\overline{t}}}
\newcommand\ou{\ensuremath{\overline{u}}}

\newcommand\ox{\ensuremath{\overline{x}}}
\newcommand\oy{\ensuremath{\overline{y}}}
\newcommand\oz{\ensuremath{\overline{z}}}

\newcommand\ophi{\ensuremath{\overline{\phi}}}
\newcommand\opsi{\ensuremath{\overline{\psi}}}

\newcommand\ux{\ensuremath{\underline{x}}}

\newcommand\hS{\ensuremath{\widehat{S}}}
\newcommand\hT{\ensuremath{\widehat{T}}}

\newcommand\hi{\ensuremath{\widehat{i}}}
\newcommand\hsi{\ensuremath{\hspace{2pt}\widehat{i}}}
\newcommand\his{\ensuremath{\widehat{i}\hspace{2pt}}}
\newcommand\hsis{\ensuremath{\hspace{2pt}\widehat{i}\hspace{2pt}}}

\newcommand\hips{\ensuremath{\widehat{i'}\hspace{2pt}}}

\newcommand\hj{\ensuremath{\widehat{j}}}
\newcommand\hsj{\ensuremath{\hspace{2pt}\widehat{j}}}
\newcommand\hjs{\ensuremath{\widehat{j}\hspace{2pt}}}
\newcommand\hsjs{\ensuremath{\hspace{2pt}\widehat{j}\hspace{2pt}}}

\newcommand\hks{\ensuremath{\widehat{k}\hspace{2pt}}}
\newcommand\hsks{\ensuremath{\hspace{2pt}\widehat{k}\hspace{2pt}}}

\newcommand\hy{\ensuremath{\widehat{y}}}

\newcommand\hSigma{\ensuremath{\widehat{\Sigma}}}
\newcommand\hiota{\ensuremath{\widehat{\iota}}}

\newdimen\CdotAxis
\newcommand*{\CdotAux}[3]{%
  {%
    \settoheight\CdotAxis{$#2\vcenter{}$}%
    \sbox0{%
      \raisebox\CdotAxis{%
        \scalebox{#1}{%
          \raisebox{-\CdotAxis}{%
            $\mathsurround=0pt #2#3$%
          }%
        }%
      }%
    }%
    \dp0=0pt %
    \sbox2{$#2\bullet$}%
    \ifdim\ht2<\ht0 %
      \ht0=\ht2 %
    \fi
    \sbox2{$\mathsurround=0pt #2#3$}%
    \hbox to \wd2{\hss\usebox{0}\hss}%
  }%
}

\newcommand\arf{\ensuremath{\operatorname{Arf}}}
\newcommand\sig{\ensuremath{\operatorname{sig}}}
\newcommand\Mod{\ensuremath{\operatorname{Mod}}}
\newcommand\QTorelli{\ensuremath{\operatorname{\cQ\cI}}}
\newcommand\Rel{\ensuremath{\operatorname{\cR}}}
\newcommand\KRel{\ensuremath{\operatorname{\cK\cR}}}
\newcommand\QK{\ensuremath{\operatorname{\cQ\cK}}}
\newcommand\Pres{\ensuremath{\operatorname{\fB}}}
\newcommand\Torelli{\ensuremath{\cI}}
\newcommand\link{\ensuremath{\operatorname{lk}}}
\newcommand\BCJ{\ensuremath{\operatorname{B}}}
\newcommand\BCJZ{\ensuremath{\operatorname{B}^0}}
\newcommand\SymSub[1]{\ensuremath{\llbracket #1 \rrbracket}}
\newcommand\uSymSub[1]{\ensuremath{\underline{\llbracket #1 \rrbracket}}}
\newcommand\LGen[1]{\ensuremath{\langle\!\langle #1 \rangle\!\rangle}}
\newcommand\igeom{\ensuremath{i_{\text{geom}}}}

\title{The kernel of the Birman--Craggs--Johnson homomorphism}

\author{Tara E. Brendle}
\address{School of Mathematics \& Statistics; University of Glasgow; University Place; Glasgow G12 8QQ; UK}
\email{tara.brendle@glasgow.ac.uk}

\author{Dan Margalit}
\address{Department of Mathematics; Vanderbilt University; 1326 Stevenson Center Ln, Nashville, TN 37240; USA}
\email{dan.margalit@vanderbilt.edu}

\author{Andrew Putman}
\address{Department of Mathematics; University of Notre Dame; 255 Hurley Hall; Notre Dame, IN 46556; USA}
\email{andyp@nd.edu}

\thanks{DM was supported by NSF grant DMS-2417920.  AP was supported by NSF grant DMS-2305183.}

\begin{document}
    
\newpage
        
\begin{abstract}
For surfaces of genus $g \geq 3$ with at most one boundary component, we 
give explicit generating sets for the kernel of the Birman--Craggs--Johnson homomorphism and the commutator subgroup
of the Torelli group. 
As an application, we give a new proof of Johnson's calculation of the abelianization of
the Torelli group.
\end{abstract} 

\maketitle
\thispagestyle{empty}

\vspace{-25pt}
\section{Introduction}
\label{section:introduction}

Let $\Sigma_g^b$ be a compact oriented genus $g$ surface with $b \in \{0,1\}$ boundary components and let $\Mod_g^b$
be its mapping class group, i.e., the group of isotopy classes of orientation-preserving diffeomorphisms
of $\Sigma_g^b$ that fix $\partial \Sigma_g^b$ pointwise.  These isotopies must also fix $\partial \Sigma_g^b$.
We will omit $b$ from our notation when $b=0$.
The group $\Mod_g^b$ acts on $\HH_1(\Sigma_g^b) \cong \Z^{2g}$ and preserves
the algebraic intersection form.  This gives a homomorphism $\Mod_g^b \rightarrow \Sp_{2g}(\Z)$ whose
kernel $\Torelli_g^b$ is the Torelli group.  These groups fit into the short exact sequence
\[1 \longrightarrow \Torelli_g^b \longrightarrow \Mod_g^b \longrightarrow \Sp_{2g}(\Z) \longrightarrow 1.\]
Building on seminal work of Birman--Craggs \cite{BirmanCraggs}, Johnson \cite{JohnsonBCJ} used the Rochlin
invariant of integral homology $3$-spheres to construct the Birman--Craggs--Johnson (BCJ) homomorphisms $\sigma\colon \Torelli_g^1 \rightarrow \BCJ_3(g)$ and
$\sigma\colon \Torelli_g \rightarrow \BCJZ_3(g)$.  Here $\BCJ_3(g)$ and $\BCJZ_3(g)$ are elementary abelian $2$-groups we will
describe in more detail below.  Johnson proved in \cite{Johnson3} that for $g \geq 3$ the BCJ homomorphisms induce
isomorphisms
\[\HH_1(\Torelli_g^1;\bbF_2) \cong \BCJ_3(g) \quad \text{and} \quad \HH_1(\Torelli_g;\bbF_2) \cong \BCJZ_3(g).\]
Johnson combined this with his previous work on the Johnson homomorphism to calculate
$(\Torelli_g^b)^{\text{ab}} = \HH_1(\Torelli_g^b)$.  We will describe this calculation in more detail below.

\subsection{Goal}
In this paper, we give an explicit generating set for the kernel of the BCJ homomorphism.
This can be viewed as a mod-$2$ analogue of another theorem of Johnson \cite{Johnson2} saying that
the kernel of the Johnson homomorphism is generated by Dehn twists about separating curves.  Using our
theorem, we give a new proof that the BCJ homomorphism detects $\HH_1(\Torelli_g^b;\bbF_2)$, and
thus a new approach to calculating $\HH_1(\Torelli_g^b)$.  We also obtain
an explicit generating set for the commutator subgroup of $\Torelli_g^b$.

\subsection{Rochlin invariant}

The Rochlin invariant (\cite{MilnorNotes}; see also \cite[Chapter XI]{Kirby4D}) is an invariant $\mu(M^3) \in \bbF_2$
of an oriented integral homology $3$-sphere $M^3$.  It can be calculated as follows.  
Since $M^3$ is
parallelizable and $\HH^1(M^3;\bbF_2)=0$, it has a unique spin structure.
Rochlin (\cite{Rochlin1, Rochlin2, RochlinCommentary}; see also \cite[Chapter VII]{Kirby4D})
proved that there exists a smooth compact spin $4$-manifold $W^4$ with $\partial W^4 = M^3$.
For number-theoretic reasons,\footnote{Since $W^4$ is spin its intersection
form is even, and since $\partial W^4$ is an integral
homology $3$-sphere its intersection form is unimodular.  We can thus invoke van der Blij's 
Lemma \cite[Corollary 9.31]{GersteinQuadratic}, which says that the signature of
a unimodular even quadratic form is divisible by $8$.}
its signature $\sig(W^4)$ is divisible by $8$.  The
Rochlin invariant of $M^3$ is
$\mu(M^3) \equiv \sig(W^4)/8$ modulo $2$.  

This is independent of $W^4$.  Indeed, if $V^4$ is another
smooth compact spin $4$-manifold with $\partial V^4 = M^3$, then $X^4 = W^4 \cup_{M^3} (-V^4)$ is
a closed smooth spin $4$-manifold with $\sig(X^4) = \sig(W^4)-\sig(V^4)$.  Another
theorem of Rochlin (\cite{Rochlin2, RochlinCommentary}; see also \cite[Chapter XI]{Kirby4D}) 
says that $\sig(X^4)$ is divisible by $16$,
so $\sig(W^4)/8 \equiv \sig(V^4)/8$ modulo $2$.

\subsection{Birman--Craggs homomorphism}
\label{section:birmancraggs}

Let $\iota\colon \Sigma_g \hookrightarrow \bbS^3$ be an embedding with $\iota(\Sigma_g)$ a 
Heegaard surface,\footnote{Waldhausen \cite{WaldhausenUnique} proved that any two 
genus-$g$ Heegaard surfaces
in $\bbS^3$ are isotopic as unparameterized surfaces, but for $g \geq 1$ there are multiple isotopy classes of 
parameterized Heegaard surfaces.} so $\bbS^3 = \cH \cup_{\iota(\Sigma_g)} \cH'$ for genus-$g$ handlebodies $\cH$ and $\cH'$.
Let $\phi\colon \partial \cH \rightarrow \partial \cH'$ be the diffeomorphism associated to this gluing,\footnote{Note that
$\phi$ is necessarily orientation-reversing.}
so $\bbS^3 = \cH \cup_{\iota(\Sigma_g)} \cH' = \cH \cup_{\phi} \cH'$.  Regard $\iota$ as a diffeomorphism
from $\Sigma_g$ to $\partial \cH$.  For 
$f \in \Mod_g$ represented by a diffeomorphism $F\colon \Sigma_g \rightarrow \Sigma_g$, we have
a diffeomorphism $\iota \circ F \circ \iota^{-1}\colon \partial \cH \rightarrow \partial \cH$.  Set
$M_{\iota,f} = \cH \cup_{\phi \circ (\iota \circ F \circ \iota^{-1})} \cH'$. 
The diffeomorphism type of the closed oriented $3$-manifold $M_{\iota,f}$ only depends on $\iota$ and $f$.  

For $f \in \Torelli_g$, the Mayer--Vietoris exact sequence implies that 
$M_{\iota,f}$ is an integral homology $3$-sphere.
We can thus define a set map $\mu_{\iota}\colon \Torelli_g \rightarrow \bbF_2$ via the formula
$\mu_{\iota}(f) = \mu(M_{\iota,f})$ for $f \in \Torelli_g$.
Birman--Craggs \cite{BirmanCraggs} proved that $\mu_{\iota}$ is a homomorphism.  

\subsection{Birman--Craggs--Johnson (BCJ) homomorphism}
\label{section:bcjintro}

Johnson \cite{JohnsonBCJ} characterized how $\mu_{\iota}$ depends on $\iota$ and used this to package
all the different $\mu_{\iota}$ together into a single homomorphism.  He also showed how to construct an
analogue of this homomorphism for $\Torelli_g^1$.  We give a brief
description of this homomorphism here; see \S \ref{section:bcjhomomorphism} below for more details.

Fix $g \geq 3$ and a symplectic basis $S=\{a_1,b_1,\ldots,a_g,b_g\}$ for $H=\HH_1(\Sigma_g^1;\bbF_2) \cong \bbF_2^{2g}$.
Regard the elements of $S$ as formal variables.
For $n \geq 0$, let $\BCJ_n(g)$ be the $\bbF_2$-vector space of all square-free polynomials $\phi \in \bbF_2[S]$ such that
the degree of $\phi$ is at most $n$.  Here by ``square-free'' we mean that none of the monomials
appearing in $\phi$ have any variable $s \in S$ with exponent greater than $1$.  For instance,
$a_1 b_1 a_2 \in \BCJ_3(g)$ but $a_1 b_1^2 \notin \BCJ_3(g)$.

The Birman--Craggs--Johnson (BCJ) homomorphism is a surjective homomorphism $\sigma\colon \Torelli_g^1 \rightarrow \BCJ_3(g)$.
It is natural in the sense that 
the conjugation action of $\Mod_g^1$ on $\Torelli_g^1$ descends to an action of
$\Sp_{2g}(\Z) \cong \Mod_g^1/\Torelli_g^1$ on $\BCJ_3(g)$ that factors through $\Sp_{2g}(\bbF_2)$.  We warn the reader 
that this is not the obvious action coming from the action of $\Sp_{2g}(\bbF_2)$ on $H$ (see \S \ref{section:bcjhomomorphism}).
On a closed surface, the BCJ homomorphism is a surjective homomorphism $\sigma\colon \Torelli_g \rightarrow \BCJZ_3(g)$, where
$\BCJZ_n(g)$ is a quotient of $\BCJ_n(g)$ we will describe in \S \ref{section:bcjhomomorphism}.

\begin{remark}
The above description of $\BCJ_3(g)$ depends on the choice of
symplectic basis $S=\{a_1,b_1,\ldots,a_g,b_g\}$ for $\HH_1(\Sigma_g^1;\bbF_2)$.
In \S \ref{section:bcjhomomorphism} we will give a basis-free description of it
that will make the action of $\Sp_{2g}(\bbF_2)$ clear.
\end{remark}

\begin{remark}
The dimension of the $\bbF_2$-vector space $\BCJ_3(g)$ is
\[\binom{2g}{0} + \binom{2g}{1} + \binom{2g}{2} + \binom{2g}{3} = \frac{1}{3} (4g^3+5g+3).\]
To form $\BCJZ_3(g)$, we quotient $\BCJ_3(g)$ by a subspace that turns out to be $(2g+1)$-dimensional.  It follows
that the dimension of $\BCJZ_3(g)$ is
$(4g^3+5g+3)/3 - (2g+1) = (4g^3-g)/3$.
\end{remark}

\subsection{Separating twists and bounding pair (BP) maps}

For a simple closed curve $x$ on $\Sigma_g^b$, let
$T_x \in \Mod_g^b$ be the left Dehn twist about $x$.  This is nontrivial if and only if $x$ is nontrivial,
i.e., does not bound a disk.  A {\em separating twist} is a Dehn twist
$T_x$ about a nontrivial simple closed separating curve $x$.  A {\em bounding pair (BP) map} is a product $T_y T_z^{-1}$ with $y$ and $z$ disjoint nonisotopic
nonseparating curves on $\Sigma_g^b$ such that $y \cup z$ separates $\Sigma_g^b$:
\Figure{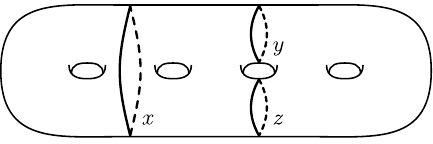}
Separating twists and BP maps lie in $\Torelli_g^b$.
Building on work of Birman \cite{BirmanSiegel}, Powell \cite{PowellTorelli} proved that $\Torelli_g^b$
is generated by separating twists and BP maps.  See \cite{HatcherMargalit, PutmanCutPaste} for
alternate proofs.  For $g \geq 3$, Johnson \cite{Johnson1} later proved that $\Torelli_g^b$ is actually
generated by finitely many BP maps.  Johnson's generating set is enormous.  See \cite{PutmanSmallGenset} for
a smaller one.

\subsection{Elements of the kernel}

Let $\sigma$ be the BCJ homomorphism on $\Torelli_g^b$.  We now describe some elements of $\ker(\sigma)$.
Since the target of $\sigma$ is an elementary abelian $2$-group, all squares of elements
of $\Torelli_g^b$ lie in its kernel.  In particular, squares of separating
twists and squares of BP maps lie in $\ker(\sigma)$.  

Next, all commutators of elements of $\Torelli_g^b$ lie in $\ker(\sigma)$.  For a group $G$ and $x,y \in G$,
our commutator convention is $[x,y] = x y x^{-1} y^{-1}$.  Let $\igeom(-,-)$ be the geometric intersection number.
A {\em basic commutator} in
$\Torelli_g^b$ is an element of the form $[T_x,T_y T_z^{-1}] \in \ker(\sigma)$, where:
\begin{itemize}
\item[(a)] $T_x$ is a separating twist; and
\item[(b)] $T_y T_z^{-1}$ is a BP map; and
\item[(c)] $\igeom(x,y)=\igeom(x,z)=2$; and
\item[(d)] putting $y$ and $z$ into minimal position with $x$, both $x \cup y$ and $x \cup z$ separate $\Sigma_g^b$ into two subsurfaces.
\end{itemize}
Here is an example (left) and a non-example (right):
\Figure{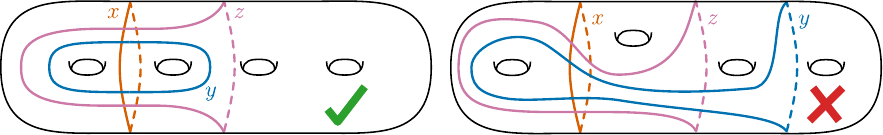}
The non-example on the right-hand side satisfies (a)-(c) but not (d).

For our final family of elements of $\ker(\sigma)$, recall that a {\em pair of pants}
is a $3$-holed sphere.  A {\em tri-separating pants map} is a product $T_x T_y T_z$, where
$x$ and $y$ and $z$ are disjoint nontrivial separating curves such that $x \cup y \cup z$ bounds
a pair of pants:
\Figure{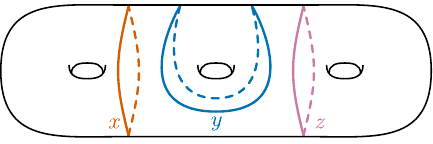}
Since $T_x$ and $T_y$ and $T_z$ all commute, their order does not matter.
We will prove in \S \ref{section:separatingpants} that tri-separating 
pants maps lie in $\ker(\sigma)$; see Lemma~\ref{lemma:seppantsker}.

\subsection{Generation theorems}

Our first main theorem says that the above elements generate the kernel of the BCJ homomorphism:

\begin{maintheorem}[Kernel of BCJ]
\label{maintheorem:genker}
For $g \geq 3$ and $b \in \{0,1\}$, the kernel of the BCJ homomorphism
on $\Torelli_g^b$ is generated by squares of separating twists, squares of BP maps, basic
commutators, and tri-separating pants maps.
\end{maintheorem}

Since the target of the BCJ homomorphism is an abelian group, its kernel contains 
the commutator subgroup.  Our second main theorem says that
the commutator subgroup is generated by all the elements used in Theorem~\ref{maintheorem:genker} except
for squares of BP maps:

\begin{maintheorem}[Commutator subgroup]
\label{maintheorem:gencomm}
For $g \geq 3$ and $b \in \{0,1\}$, the group $[\Torelli_g^b,\Torelli_g^b]$
is generated by squares of separating twists, basic
commutators, and tri-separating pants maps.
\end{maintheorem}

\begin{remark}
We will be careful to separate the proof of Theorem~\ref{maintheorem:gencomm} from the
proof of Theorem~\ref{maintheorem:genker}.  This is important since our proof
of Theorem~\ref{maintheorem:gencomm}
depends on Johnson's calculation of $\HH_1(\Torelli_g^b)$.  If we intertwined
the proofs of Theorems~\ref{maintheorem:genker} and \ref{maintheorem:gencomm},
we would risk making the new proof of Johnson's calculation of $\HH_1(\Torelli_g^b)$ described
below circular.
\end{remark}

\subsection{Mod-2 commutator subgroup of the Torelli group}

Johnson \cite{Johnson3} proved that for $g \geq 3$ and $b \in \{0,1\}$ the
BCJ homomorphism detects $\HH_1(\Torelli_g^b;\bbF_2)$, so
\[\HH_1(\Torelli_g^1;\bbF_2) \cong \BCJ_3(g) \quad \text{and} \quad \HH_1(\Torelli_g;\bbF_2) \cong \BCJZ_3(g).\]
Theorem~\ref{maintheorem:genker} can be used to give a new proof of this in the
following way.  For a group $G$, the {\em mod-$2$ commutator subgroup} of $G$, denoted $[G,G]_{\bbF_2}$, is the kernel
of the natural map $G \rightarrow \HH_1(G;\bbF_2)$.  It is generated\footnote{In fact, 
for $g_1,g_2 \in G$ we have $[g_1,g_2] = g_1^2(g_1^{-1} g_2)^2 g_2^{-2}$.  
It follows that $[G,G]_{\bbF_2}$ is generated by squares.}  by commutators $[g_1,g_2]$
and squares $g^2$ with $g,g_1,g_2 \in G$.  Squares of separating twists,
squares of BP maps, and basic commutators therefore all lie in $[\Torelli_g^b,\Torelli_g^b]_{\bbF_2}$.  We will also prove
that tri-separating pants maps lie in $[\Torelli_g^b,\Torelli_g^b]_{\bbF_2}$; see Lemma~\ref{lemma:seppantscom}.

Fix $g \geq 3$ and $b \in \{0,1\}$, and let $\sigma$ be the BCJ homomorphism on $\Torelli_g^b$.  Theorem~\ref{maintheorem:genker}
and the above discussion imply that $\ker(\sigma) \subset [\Torelli_g^b,\Torelli_g^b]_{\bbF_2}$.  Conversely,
since the target of $\sigma$ is an $\bbF_2$-vector space we have $[\Torelli_g^b,\Torelli_g^b]_{\bbF_2} \subset \ker(\sigma)$.
We deduce:

\begin{maincorollary}[{Johnson \cite{Johnson3}}]
\label{maincorollary:torellimod2}
For $g \geq 3$ and $b \in \{0,1\}$, the kernel of the BCJ homomorphism is
$[\Torelli_g^b,\Torelli_g^b]_{\bbF_2}$.  Consequently,
the BCJ homomorphisms induce isomorphisms
$\HH_1(\Torelli_g^1;\bbF_2) \cong \BCJ_3(g)$ and $\HH_1(\Torelli_g;\bbF_2) \cong \BCJZ_3(g)$.
\end{maincorollary}

\subsection{Johnson homomorphism}
\label{section:johnsonhomo}

Fix $g \geq 3$ and $b \in \{0,1\}$.
We now describe how Johnson used Corollary~\ref{maincorollary:torellimod2} to calculate $\HH_1(\Torelli_g^b)$.  The
remaining ingredient is the Johnson homomorphism.  
Let $H_{\Z} = \HH_1(\Sigma_g^b) \cong \Z^{2g}$.  We call this $H_{\Z}$ since in this paper $H$ will always mean $\HH_1(\Sigma_g;\bbF_2) \cong \bbF_2^{2g}$.  The Johnson homomorphisms are surjective homomorphisms
\[\tau\colon \Torelli_g^1 \rightarrow \wedge^3 H_{\Z} \quad \text{and} \quad \tau\colon \Torelli_g \rightarrow (\wedge^3 H_{\Z})/H_{\Z}.\]
Here $H_{\Z}$ is embedded in $\wedge^3 H_{\Z}$ via the map $h \mapsto h \wedge \omega$, where
$\omega \in \wedge^2 H_{\Z} \cong \wedge^2 H_{\Z}^{\ast}$ is the symplectic form.
See \cite{JohnsonHomo} for the original construction, \cite{JohnsonSurvey} for a survey of applications and
other constructions, and \cite[\S 6.6]{FarbMargalitPrimer} for a textbook reference.
The {\em Johnson kernel} is the kernel $\cK_g^b$ of the Johnson homomorphism $\tau$.  We therefore have short
exact sequences
\[\begin{tikzcd}[column sep=small, row sep=0pt]
1 \arrow{r} & \cK_g^1 \arrow{r} & \Torelli_g^1 \arrow{r} & \wedge^3 H_{\Z} \arrow{r} & 1, \\
1 \arrow{r} & \cK_g   \arrow{r} & \Torelli_g   \arrow{r} & (\wedge^3 H_{\Z})/H_{\Z}   \arrow{r} & 1.
\end{tikzcd}\]
Since the target
of $\tau$ is abelian, we have $[\Torelli_g^b,\Torelli_g^b] \subset \cK_g^b$.  
Johnson (\cite{Johnson2}, see \cite{PutmanJohnson} for an alternate proof) 
showed that $\cK_g^b$ is generated by separating twists.\footnote{One can view Theorem~\ref{maintheorem:genker} as
an analogue of this for the BCJ homomorphism.}  In addition, for $g \geq 3$ Johnson showed in \cite{Johnson2} that for a separating
twist $T_x$ we have $T_x^2 \in [\Torelli_g^b,\Torelli_g^b]$.  This calculation also appears in \cite{PutmanJohnson}.  It
implies that $\cK_g^b/[\Torelli_g^b,\Torelli_g^b]$ is an elementary abelian $2$-group.  Moreover, since
the target of $\tau$ is free abelian we deduce that the torsion subgroup
of $\HH_1(\Torelli_g^b)$ is exactly the elementary abelian $2$-group $\cK_g^b/[\Torelli_g^b,\Torelli_g^b]$.

\subsection{Abelianization}
\label{section:torelliabelianization}

Combined with Corollary~\ref{maincorollary:torellimod2}, the above discussion implies that $\HH_1(\Torelli_g^b)$ is detected
by the Johnson homomorphism $\tau$ (which detects the torsion-free part) and the BCJ homomorphism $\sigma$ (which detects the $2$-torsion).
Moreover, the torsion subgroup of $\HH_1(\Torelli_g^b)$ is the image of $\cK_g^b$ under the map
$\cK_g^b \rightarrow \HH_1(\Torelli_g^b)$.  
Johnson proved that $\sigma(\cK_g^1) = \BCJ_2(g)$ and $\sigma(\cK_g)=\BCJZ_2(g)$.  Also, letting
$H = \HH_1(\Sigma_g^b;\bbF_2)$ we can identify $\BCJ_3(g)/\BCJ_2(g)$ with $\wedge^3 H$ and
$\BCJZ_3(g)/\BCJZ_2(g)$ with $(\wedge^3 H)/H$.  Making these identifications, Johnson proved that there are commutative 
diagrams
\[\begin{tikzcd}[column sep=tiny]
1 \arrow{r} & \cK_g^1 \arrow[two heads]{d}{\sigma} \arrow{r} & \Torelli_g^1 \arrow[two heads]{d}{\sigma} \arrow{r}{\tau} & \wedge^3 H_{\Z} \arrow[two heads]{d}{\substack{\text{reduce} \\ \text{mod } 2}} \arrow{r} & 1 \\
0 \arrow{r} & \BCJ_2(g)  \arrow{r}                              & \BCJ_3(g) \arrow{r}    & \BCJ_3(g)/\BCJ_2(g) \arrow{r} & 0
\end{tikzcd}
\enspace \text{and} \enspace
\begin{tikzcd}[column sep=tiny]
1 \arrow{r} & \cK_g \arrow[two heads]{d}{\sigma} \arrow{r} & \Torelli_g \arrow[two heads]{d}{\sigma} \arrow{r}{\tau} & (\wedge^3 H_{\Z})/H_{\Z} \arrow[two heads]{d}{\substack{\text{reduce} \\ \text{mod } 2}} \arrow{r} & 1 \\
0 \arrow{r} & \BCJZ_2(g)  \arrow{r}                          & \BCJZ_3(g) \arrow{r}  & \BCJZ_3(g)/\BCJZ_2(g) \arrow{r} & 0.
\end{tikzcd}\]
From this, we see that $\BCJ_2(g)$ and $\BCJZ_2(g)$ are isomorphic to the $2$-torsion in $\HH_1(\Torelli_g^1)$ and $\HH_1(\Torelli_g)$, respectively.  We conclude:

\begin{maincorollary}[{Johnson \cite{Johnson3}}]
\label{maincorollary:h1torelli}
For $g \geq 3$, we have $\HH_1(\Torelli_g^1) \cong \BCJ_2(g) \oplus \wedge^3 H_{\Z}$ and
$\HH_1(\Torelli_g) \cong \BCJZ_2(g) \oplus (\wedge^3 H_{\Z})/H_{\Z}$.
\end{maincorollary}

\begin{remark}
The direct sum decompositions of $\HH_1(\Torelli_g^b)$ in Corollary \ref{maincorollary:h1torelli} are not canonical.
One can show there are no such decompositions that are preserved by the action of $\Sp_{2g}(\Z)$.
\end{remark}

\subsection{Proof ideas and outline of the paper}
\label{section:proofideas}

Our proofs of Theorems~\ref{maintheorem:genker} and \ref{maintheorem:gencomm} use a mixture of topological and representation-theoretic reasoning.
For simplicity, we restrict to the special case of surfaces with one boundary
component.  

In this case, recall that Theorem~\ref{maintheorem:genker} says that the 
kernel of the BCJ homomorphism $\sigma\colon \Torelli_g^1 \rightarrow \BCJ_3(g)$
is generated by squares of separating twists, squares of BP maps, basic
commutators, and tri-separating pants maps.  These purported generators lie in $\ker(\sigma)$.
The subgroup $\Rel_g^1$ of $\Torelli_g^1$ they generate is normal since each family of generators
is invariant under conjugation by $\Mod_g^1$.  Let $\QTorelli_g^1 = \Torelli_g^1/\Rel_g^1$.
There is an induced map $\QTorelli_g^1 \rightarrow \BCJ_3(g)$.  Theorem~\ref{maintheorem:genker}
can be restated as saying that this induced map is an isomorphism.  

Similarly, Theorem~\ref{maintheorem:gencomm} says that $[\Torelli_g^1,\Torelli_g^1]$
is generated by squares of separating twists, basic commutators, and
tri-separating pants maps.  As we discussed when we sketched Johnson's calculation
of $\HH_1(\Torelli_g^1)$, it follows from Johnson's work that $[\Torelli_g^1,\Torelli_g^1] \subset \cK_g^1$
and that $[\Torelli_g^1,\Torelli_g^1]$ is the kernel of the BCJ homomorphism
$\sigma\colon \cK_g^1 \rightarrow \BCJ_2(g)$.  

Let $\KRel_g^1$ be the subgroup
of $\cK_g^1$ generated by squares of separating twists, basic commutators, and
tri-separating pants maps.  Just like for $\Rel_g^1$, this is a normal subgroup.
Set $\QK_g^1 = \cK_g^1/\KRel_g^1$.  There is an induced map
$\QK_g^1 \rightarrow \BCJ_2(g)$, and Theorem~\ref{maintheorem:gencomm}
can be restated as saying that this induced map is an isomorphism.

The proof that the induced maps $\QTorelli_g^1 \rightarrow \BCJ_3(g)$ and
$\QK_g^1 \rightarrow \BCJ_2(g)$ are isomorphisms has three parts:

\begin{itemize}[leftmargin=50pt]
\item[{\bf Part \ref{part:1}}.]
The first step will be to show that it is enough to just prove that
the induced map $\QK_g^1 \rightarrow \BCJ_2(g)$ is an isomorphism.
In this step, we also show that the analogous results for closed surfaces
follow from the fact that the map $\QK_g^1 \rightarrow \BCJ_2(g)$ is an isomorphism.
The remaining parts of the paper therefore only consider surfaces with boundary.

\item[{\bf Part \ref{part:2}}.]
The second step establishes properties of $\QK_g^1$ that mimic properties of $\BCJ_2(g)$:
\begin{itemize}
\item $\QK_g^1$ is an elementary abelian $2$-group and hence an $\bbF_2$-vector space; and
\item $\QK_g^1$ has a natural action of $\Sp_{2g}(\bbF_2)$.
\end{itemize}
These two properties can be summarized as saying that
$\QK_g^1$ is a representation of $\Sp_{2g}(\bbF_2)$ over the field $\bbF_2$.
The arguments in this step use
ideas from van den Berg's unpublished PhD thesis \cite{VanDenBergThesis}.

\item[{\bf Part \ref{part:3}}.]
The final step constructs generators and relations for $\QK_g^1$ as a representation
of $\Sp_{2g}(\bbF_2)$.  We then show that these generators and relations imply an
isomorphism $\QK_g^1 \cong \BCJ_2(g)$ of $\Sp_{2g}(\bbF_2)$-representations.  We use
two key tools:
\begin{itemize}
\item The notion of central stability from
Church--Farb's theory of representation stability \cite{ChurchFarbRepStability}, which
was introduced by Putman \cite{PutmanCongruence} and further developed by Putman--Sam \cite{PutmanSamNoetherian}.
\item Recent work of Minahan--Putman \cite{MinahanPutmanRep} giving
a framework for establishing presentations of representations of this sort.  
\end{itemize}
To verify the genus-$3$ base case, we rely on a computer calculation.
\end{itemize}
The paper is divided into three parts corresponding to the three parts of the proof described
above.  We organize it as a sequence of reductions: Part~\ref{part:1} reduces 
Theorems~\ref{maintheorem:genker} and \ref{maintheorem:gencomm} to 
Theorem~\ref{maintheorem:qkg}, Part~\ref{part:2} reduces Theorem~\ref{maintheorem:qkg} to Theorem~\ref{maintheorem:qkgpres},
and Part~\ref{part:3} proves Theorem~\ref{maintheorem:qkgpres}.  

\part{Preliminary results and a reduction to a quotient of the Johnson kernel}
\label{part:1}

We start in \S \ref{section:bcjhomomorphism} by giving more details about the BCJ homomorphism.
Next, in \S \ref{section:separatingpants} we prove that tri-separating pants maps lie in both the
kernel of the BCJ homomorphism and the mod-$2$ commutator subgroup of the Torelli group.
We then discuss some facts about basic commutators in \S \ref{section:simplebasic}.
Finally, in \S \ref{section:quotients} we reduce Theorem~\ref{maintheorem:genker} to Theorem~\ref{maintheorem:qkg},
which is a result about a certain quotient of the Johnson kernel.  The derivation of
Theorem~\ref{maintheorem:gencomm} from Theorem~\ref{maintheorem:qkg} is also in \S \ref{section:quotients}.

\begin{genusassumption}
Throughout Part~\ref{part:1}, we will assume that the genus $g$ satisfies $g \geq 2$.
Near the end of this part we will need to increase this to $g \geq 3$, and we will be careful
to be explicit about this.
\end{genusassumption}

\section{Preliminaries on the Birman--Craggs--Johnson (BCJ) homomorphism}
\label{section:bcjhomomorphism}

Fix $g \geq 2$.  This section surveys Johnson's construction of the BCJ homomorphism \cite{JohnsonBCJ}.

\subsection{Quadratic forms}

Let $H = \HH_1(\Sigma_g;\bbF_2) = \HH_1(\Sigma_g^1;\bbF_2)$ and let $\hiota(-,-)$ be the $\bbF_2$-valued
algebraic intersection form on $H$.  A {\em quadratic form} on $H$ is a set
map $q\colon H \rightarrow \bbF_2$ such that
\[q(x+y) = q(x) + q(y) + \hiota(x,y) \quad \text{for all $x,y \in H$}.\]
If $S = \{a_1,b_1,\ldots,a_g,b_g\}$ is a symplectic
basis for $H$, then for all $\lambda_1,\lambda'_1,\ldots,\lambda_g,\lambda'_g \in \bbF_2$ there exists
a unique quadratic form $q\colon H \rightarrow \bbF_2$ with
$q(a_i) = \lambda_i$ and $q(b_i) = \lambda'_i$ for all $1 \leq i \leq g$.
Let $\Omega_g$ be the set of quadratic forms on $H$.  The set $\Omega_g$ is not closed under
addition.  However, let $H^{\ast} = \Hom(H,\bbF_2)$ be the dual
of $H$.  For a fixed $q_0 \in \Omega_g$, we then have
$\Omega_g = \Set{$q_0 + \phi$}{$\phi \in H^{\ast}$}$.
In other words, $\Omega_g$ is a torsor for $H^{\ast}$.

\subsection{Arf invariant}

A textbook reference for this material is \cite[Chapter III.1]{BrowderSurgery}.
Consider $q \in \Omega_g$.  The {\em Arf invariant} of $q$, denoted $\arf(q)$, is defined as
follows.  Choose a symplectic basis $\{a_1,b_1,\ldots,a_g,b_g\}$ for $H$.  Then
\[\arf(q) = q(a_1) q(b_1) + \cdots + q(a_g) q(b_g).\]
It is a standard fact that $\arf(q)$ does not depend on the choice of the symplectic basis.
For $c \in \bbF_2$, let $\Omega_g(c) = \Set{$q \in \Omega_g$}{$\arf(q) = c$}$.  
The group $\Sp_{2g}(\bbF_2)$ acts on $\Omega_g$ as follows:
\begin{equation}
\label{eqn:spaction1}
(Mq)(h) = q(M^{-1} h) \quad \text{for $M \in \Sp_{2g}(\bbF_2)$ and $q \in \Omega_g$ and $h \in H$}.
\end{equation}
This action preserves the Arf invariant, and $\Sp_{2g}(\bbF_2)$ acts transitively on both $\Omega_g(0)$ and $\Omega_g(1)$.
In other words, the Arf invariant is a complete invariant of the isomorphism
class of a quadratic form.

\subsection{Polynomial functions}

Let $\Omega_g^{\ast}$ be the ring of functions $\mu\colon \Omega_g \rightarrow \bbF_2$.  Via
its action on $\Omega_g$, the group $\Sp_{2g}(\bbF_2)$ acts on $\Omega_g^{\ast}$ as follows:
\begin{equation}
\label{eqn:spaction2}
(M\mu)(q) = \mu(M^{-1} q) \quad \text{for $M \in \Sp_{2g}(\bbF_2)$ and $\mu \in \Omega_g^{\ast}$ and $q \in \Omega_g$}.
\end{equation}
For each $c \in \bbF_2$, we have the constant function $c \in \Omega_g^{\ast}$.  Also,
for $h \in H$, we have $\oh \in \Omega_g^{\ast}$ defined by
$\oh(q) = q(h)$ for $q \in \Omega_g$.
With this notation, we have
\[\overline{h_1+h_2} = \oh_1 + \oh_2 + \hiota(h_1,h_2) \quad \text{for $h_1,h_2 \in H$}.\]
For $\mu \in \Omega_g^{\ast}$, we have $\mu^2 = \mu$.  Motivated by all of this,
let $\bbF_2[\oH]$ be the ring of polynomials in the formal variables $\Set{$\oh$}{$h \in H$}$ and let $\BCJ(g)$
be the quotient of $\bbF_2[\oH]$ by the ideal generated by:
\begin{itemize}
\item $f^2-f$ for $f \in \bbF_2[\oH]$; and
\item $\overline{h_1+h_2} - (\oh_1 + \oh_2 + \hiota(h_1,h_2))$ for all $h_1,h_2 \in H$.
\end{itemize}
The signs are not necessary since we are working over $\bbF_2$, but we include them for clarity.

For a symplectic basis $\{a_1,b_1,\ldots,a_g,b_g\}$ for $H$, elements
of $\BCJ(g)$ can be uniquely expressed as square-free polynomials in the variables
$\{\oa_1,\ob_1,\ldots,\oa_g,\ob_g\}$.  We have a map
$\BCJ(g) \rightarrow \Omega_g^{\ast}$ taking $\oh \in \oH$ to
$\oh \in \Omega_g^{\ast}$.  We claim that this is an isomorphism.  Indeed, identify $\Omega_g$
with $\bbF_2^{2g}$ by identifying $q \in \Omega_g$ with $(q(a_1),q(b_1),\ldots,q(a_g),q(b_g)) \in \bbF_2^{2g}$.
This identifies $\Omega_g^{\ast}$ with the set of functions $\bbF_2^{2g} \rightarrow \bbF_2$.
It is well-known that each such function can be uniquely represented by a square-free polynomial
in $2g$ variables, whence the claim.

Via the isomorphism $\BCJ(g) \cong \Omega_g^{\ast}$, the group $\Sp_{2g}(\bbF_2)$ acts on $\BCJ(g)$.  Combining
\eqref{eqn:spaction1} and \eqref{eqn:spaction2}, this action is as follows for monomials in $\BCJ(g)$:
\begin{equation}
\label{eqn:spaction3}
M(\oh_1\  \oh_2\  \cdots\ \oh_n) = \overline{M h_1}\ \ \overline{M h_2}\ \ \cdots\ \ \overline{M h_n}
\quad \text{for $M \in \Sp_{2g}(\bbF_2)$ and $h_1,\ldots,h_n \in H$}.
\end{equation}
The point here is that the two inverses in \eqref{eqn:spaction1} and \eqref{eqn:spaction2} cancel each
other out.

Elements of $\BCJ(g)$ do not have a well-defined
degree since the relations we impose are not homogeneous.  It therefore inherits a
filtration by degree rather than a grading.  Let $\BCJ_n(g)$ be the image of the subspace of
polynomials of degree at most $n$.  We get an increasing chain of vector spaces
\[0 = \BCJ_{-1}(g) \subsetneq \BCJ_{0}(g) \subsetneq \BCJ_{1}(g) \subsetneq \cdots \subsetneq \BCJ_{2g}(g) = \BCJ(g).\]
Let $\Omega_g(0)^{\ast}$ be the ring of functions $\mu\colon \Omega_g(0) \rightarrow \bbF_2$.
Finally, let $\BCJZ_{n}(g)$ be the image of $\BCJ_{n}(g)$ 
under the restriction map $\Omega_g^{\ast} \rightarrow \Omega_g(0)^{\ast}$.  The action
of $\Sp_{2g}(\bbF_2)$ on $\Omega_g^{\ast} \cong \BCJ(g)$ descends to an action
on $\Omega_g(0)^{\ast}$ preserving each $\BCJZ_n(g)$.  We will later need the following:

\begin{lemma}[{\cite[Lemma 14]{JohnsonBCJ}}]
\label{lemma:bcj2difference}
Let $g \geq 2$ and let $\{a_1,b_1,\ldots,a_g,b_g\}$ be a symplectic basis for
$H = \HH_1(\Sigma_g;\bbF_2)$.  Then the kernel of the map $\BCJ_2(g) \rightarrow \BCJZ_2(g)$
is isomorphic to $\bbF_2$ and is generated by
$\kappa = \oa_1 \ob_1 + \cdots + \oa_g \ob_g$.
\end{lemma}

\begin{remark}
Considered as an element of $\Omega_g^{\ast}$, the element $\kappa \in \BCJ_2(g)$ from Lemma~\ref{lemma:bcj2difference}
takes $q \in \Omega_g$ to $\arf(q) \in \bbF_2$. 
\end{remark}

\subsection{BCJ homomorphism, closed case}

Recall from \S \ref{section:birmancraggs} that the Birman--Craggs
homomorphism $\mu_{\iota}\colon \Torelli_g \rightarrow \bbF_2$ depends on the choice of an embedding
$\iota\colon \Sigma_g \hookrightarrow \bbS^3$ such that $\iota(\Sigma_g)$ is a Heegaard surface.  The self-linking
form on $\HH_1(\iota(\Sigma_g);\bbF_2)$ pulls back to a quadratic form $q_{\iota}$ on $H = \HH_1(\Sigma_g;\bbF_2)$:
\[q_{\iota}(h) = \link_{\bbF_2}(c,c^{+}) \in \bbF_2 \quad \text{for $h \in H$ and a cycle $c$ representing $\iota_{\ast}(h) \in \HH_1(\iota(\Sigma_g);\bbF_2)$}.\]
Here $c^{+}$ is the result of pushing $c$ off of $\iota(\Sigma_g) \subset \bbS^3$ in the positive normal
direction.  The Arf invariant of $q_{\iota}$ is $0$.
Johnson \cite{JohnsonBCJ} proved that $\mu_{\iota}$ is entirely determined
by $q_{\iota}$ and that every $q \in \Omega_g(0)$ equals $q_{\iota}$ for some such $\iota\colon \Sigma_g \hookrightarrow \bbS^3$.
Using this, he proved that there is a homomorphism $\sigma\colon \Torelli_g \rightarrow \Omega_g(0)^{\ast}$
taking $f \in \Torelli_g$ to the following $\sigma(f) \in \Omega_g(0)^{\ast}$:
\begin{itemize}
\item Consider $q \in \Omega_g(0)$.  Let $\iota\colon \Sigma_g \hookrightarrow \bbS^3$ be an embedding
whose image is a Heegaard surface with associated quadratic form $q_{\iota} = q$.  Then
$\sigma(f)(q) = \mu_{\iota}(f)$.
\end{itemize}
Johnson then proved that the image of $\sigma$ is exactly $\BCJZ_3(g)$.
This gives the BCJ homomorphism $\sigma\colon \Torelli_g \rightarrow \BCJZ_3(g)$ on a closed surface.

\subsection{BCJ homomorphism, surfaces with boundary}

Johnson defined $\sigma\colon \Torelli_g^1 \rightarrow \BCJ_3(g)$
as follows.  Embed $\Sigma_g^1$ into $\Sigma_{g+1}$.
We get an induced injection $\HH_1(\Sigma_g^1;\bbF_2) \hookrightarrow \HH_1(\Sigma_{g+1};\bbF_2)$.  Using
this, we can restrict quadratic forms on $\HH_1(\Sigma_{g+1};\bbF_2)$ to $\HH_1(\Sigma_g^1;\bbF_2)$.
This restriction can take a quadratic form of Arf invariant $0$ to a quadratic form of arbitrary Arf invariant,
and Johnson proved that the resulting map $\Omega_{g+1}(0) \rightarrow \Omega_g$ is a surjection.  Dualizing,
we get an injection $\Omega_g^{\ast} \hookrightarrow \Omega_{g+1}(0)^{\ast}$.  For all $n \geq 0$, this injection takes
the subspace $\BCJ_n(g)$ of $\Omega_g^{\ast} \cong \BCJ(g)$ to the subspace $\BCJZ_n(g+1)$ of
$\Omega_{g+1}(0)^{\ast}$.  

Turning to the Torelli group, by extending 
elements of $\Torelli_g^1$ to $\Sigma_{g+1}$ by the identity we get an injection $\Torelli_g^1 \hookrightarrow \Torelli_{g+1}$.
Johnson proved that the image of the composition of this injection with
$\sigma\colon \Torelli_{g+1} \rightarrow \BCJZ_3(g+1)$ 
is the image of $\BCJ_3(g)$ in $\BCJZ_3(g+1)$.  This leads to a homomorphism $\sigma\colon \Torelli_g^1 \rightarrow \BCJ_3(g)$
fitting into a commutative diagram
\[\begin{tikzcd}[column sep=small, row sep=small]
\Torelli_g^1 \arrow{r}{\sigma} \arrow[hook]{d} & \BCJ_3(g) \arrow[hook]{d} \\
\Torelli_{g+1} \arrow{r}{\sigma} & \BCJZ_3(g+1).
\end{tikzcd}\]

\subsection{Separating twists}
\label{subsection:septwistsbcj}

Let $b \in \{0,1\}$ and let $T_x$ be a separating twist.  Johnson proved that $\sigma(T_x)$ has the following
description.  Let $S$ be a subsurface of $\Sigma_g^b$ with exactly one boundary component and $\partial S = x$.
There is one choice of $S$ if $b=1$ and two choices if $b=0$:
\Figure{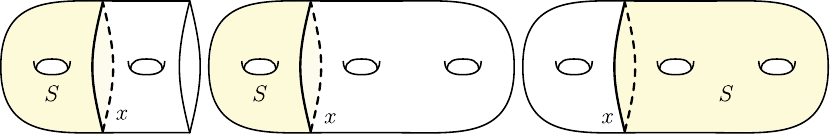}
From $S$, we have a symplectic subspace $\HH_1(S;\bbF_2)$ of $H$.  We can restrict any quadratic form $q$ on $H$ to
$\HH_1(S;\bbF_2)$, and Johnson \cite{JohnsonBCJ} proved that $\sigma(T_x)$ is the map taking
a quadratic form $q$ (of Arf invariant $0$ if $b=0$) to the Arf invariant of the restriction of
$q$ to $\HH_1(S;\bbF_2)$.  For $b=0$, there are two possible choices of $S$; however, this is
independent of the choice since if $S'$ is the subsurface on the other side of $x$ then
\[0 = \arf(q) = \arf(q|_{\HH_1(S;\bbF_2)}) + \arf(q|_{\HH_1(S';\bbF_2)}) \quad \text{for $q \in \Omega_g(0)$}.\] 
We can write down a formula for $\sigma(T_x)$ as follows:
letting $\{x_1,y_1,\ldots,x_h,y_h\}$ be a symplectic basis for $\HH_1(S;\bbF_2)$, we have
$\sigma(T_x) = \ox_1 \oy_1 + \cdots +\ox_h \oy_h$.  For closed surfaces, this
should be interpreted as an element of $\BCJZ_2(g)$.

\begin{remark}
Recalling that the Johnson kernel $\cK_g^b$ is the subgroup of the Torelli group generated
by separating twists, Johnson used this calculation to show that
$\sigma(\cK_g^1) = \BCJ_2(g)$ and $\sigma(\cK_g) = \BCJZ_2(g)$.
Johnson \cite{JohnsonBCJ} also gave a formula for the value of $\sigma$ on a BP map. 
This formula has a nonzero cubic term, and Johnson used it to prove that $\sigma$ takes
$\Torelli_g^1$ onto $\BCJ_3(g)$ and $\Torelli_g$ onto $\BCJZ_3(g)$.  We will not
need this formula.
\end{remark}

\section{Tri-separating pants maps and the BCJ homomorphism}
\label{section:separatingpants}

Fix $g \geq 2$ and $b \in \{0,1\}$.  Recall that a {\em pair of pants}
is a $3$-holed sphere.  A {\em tri-separating pants map} is a product $T_x T_y T_z \in \Torelli_g^b$, where
$x$ and $y$ and $z$ are disjoint nontrivial separating curves such that $x \cup y \cup z$ bounds
a pair of pants.  
This section proves that tri-separating pants maps lie in the kernel of the BCJ homomorphism (Lemma~\ref{lemma:seppantsker}), and then
proves that they also lie in the mod-$2$ commutator subgroup of $\Torelli_g^b$ (Lemma~\ref{lemma:seppantscom}).

\subsection{Kernel of the BCJ homomorphism}

We start by proving that tri-separating pants maps lie in the kernel of the BCJ homomorphism:

\begin{lemma}
\label{lemma:seppantsker}
Fix $g \geq 2$ and $b \in \{0,1\}$.  Let $T_x T_y T_z \in \Torelli_g^b$ be a tri-separating pants map.
Then $T_x T_y T_z$ lies in the kernel of the BCJ homomorphism. 
\end{lemma}
\begin{proof}
Isotoping $x$ and $y$ and $z$, we can assume
that all three curves lie in $\Int(\Sigma_g^b)$.  Let $P \subset \Sigma_g^b$ be the pair of
pants such that $\partial P = x \sqcup y \sqcup z$.  Let $C_x$ and $C_y$ and $C_z$ be the three components
of $\Sigma_g^b \setminus \Int(P)$, indexed such that $x \subset C_x$ and $y \subset C_y$ and $z \subset C_z$:
\Figure{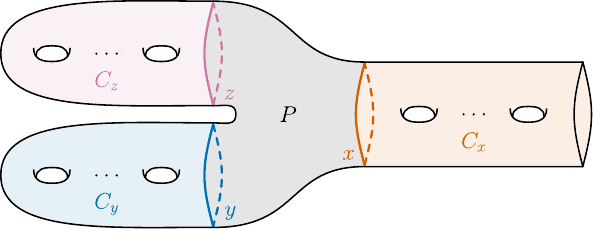}
If $b=1$, then exactly one of $C_x$ and $C_y$ and $C_z$ contains $\partial \Sigma_g^b \cong \bbS^1$.  In that
case, just like in the above figure we can reorder $x$ and $y$ and $z$ to ensure that $\partial \Sigma_g^b \subset C_x$.  Since $T_x$ and $T_y$
and $T_z$ commute, this reordering does not affect the value of $T_x T_y T_z$.

Let $S = P \cup C_y \cup C_z$.  Since $\partial \Sigma_g^b \subset C_x$ if $b = 1$, the following hold:
\begin{itemize}
\item $C_y \cong \Sigma_{h_1}^1$ and $C_z \cong \Sigma_{h_2}^1$ and $S \cong \Sigma_{h_1+h_2}^1$ with $h_1,h_2 \geq 1$; and
\item $x,y,z \subset S$ with $x = \partial S$.
\end{itemize}
We can identify $\HH_1(C_y;\bbF_2) \cong \bbF_2^{2 h_1}$ and 
$\HH_1(C_z;\bbF_2) \cong \bbF_2^{2 h_2}$ and $\HH_1(S;\bbF_2) \cong \bbF_2^{2h_1+2h_2}$
with subspaces of $\HH_1(\Sigma_g^b;\bbF_2)$.  Under this identification,
all three are symplectic subspaces.  Moreover, $\HH_1(C_y;\bbF_2)$ and
$\HH_1(C_z;\bbF_2)$ are orthogonal symplectic subspaces with
$\HH_1(S;\bbF_2) = \HH_1(C_y;\bbF_2) \oplus \HH_1(C_z;\bbF_2)$.
Let 
$\{a_1,b_1,\ldots,a_{h_1},b_{h_1}\}$ and $\{a'_{1},b'_{1},\ldots,a'_{h_2},b'_{h_2}\}$
be symplectic bases for $\HH_1(C_y;\bbF_2)$ and $\HH_1(C_z;\bbF_2)$, respectively.
Their union is a symplectic
basis for $\HH_1(S;\bbF_2)$.  Using the formulas for the BCJ homomorphism
from \S \ref{subsection:septwistsbcj}, we see that
\begin{align*}
\sigma(T_x T_y T_z) = &(\oa_1 \ob_1 + \cdots + \oa_{h_1} \ob_{h_1} + \oa'_1 \ob'_1 + \cdots + \oa'_{h_2} \ob'_{h_2}) \\
                    &+(\oa_1 \ob_1 + \cdots + \oa_{h_1} \ob_{h_1}) + (\oa'_{1} \ob'_{1} + \cdots + \oa'_{h_2} \ob'_{h_2}) \\
                  = &2 \oa_1 \ob_1 + \cdots + 2 \oa_{h_1} \ob_{h_1} + 2 \oa'_1 \ob'_1 + \cdots + 2 \oa'_{h_2} \ob'_{h_2}= 0,
\end{align*}
where the final equality uses the fact that we are working over $\bbF_2$.  The lemma follows.
\end{proof}

\subsection{Tri-separating pants maps and the mod-2 commutator subgroup}

For a group $G$, recall that 
the mod-$2$ commutator subgroup of $G$, denoted $[G,G]_{\bbF_2}$, is the kernel
of the natural map $G \rightarrow \HH_1(G;\bbF_2)$.  It is generated by commutators $[g_1,g_2]$
and squares $g^2$ with $g,g_1,g_2 \in G$. 

\begin{lemma}
\label{lemma:seppantscom}
Fix $g \geq 2$ and $b \in \{0,1\}$.  Let $T_x T_y T_z \in \Torelli_g^b$ be a tri-separating pants map.
Then $T_x T_y T_z \in [\Torelli_g^b,\Torelli_g^b]_{\bbF_2}$.
\end{lemma}
\begin{proof}
As in the proof of Lemma~\ref{lemma:seppantsker}, after possibly reordering $x$ and $y$ and $z$
we can find a subsurface $S$ of $\Sigma_g^b$ such that the following hold:
\begin{itemize}
\item $S \cong \Sigma_{h}^1$ for some $h \geq 2$; and
\item $x,y,z \subset S$ with $x = \partial S$.
\end{itemize}
Letting $\Torelli(S)$ be the Torelli group of $S$, extending by the identity gives an injective
homomorphism $\Torelli(S) \hookrightarrow \Torelli_g^b$ mapping $[\Torelli(S),\Torelli(S)]_{\bbF_2}$
into $[\Torelli_g^b,\Torelli_g^b]_{\bbF_2}$.  It is thus sufficient
to prove that $T_x T_y T_z \in [\Torelli(S),\Torelli(S)]_{\bbF_2}$.
Replacing $\Sigma_g^b$ with $S$, we can therefore assume without loss of generality that 
$b = 1$ and $x = \partial \Sigma_g^1$.

Since $T_z^2 \in [\Torelli_g^1,\Torelli_g^1]_{\bbF_2}$ and
$T_x T_y T_z = (T_x T_y T_z^{-1}) T_z^2$, it is enough to prove that $T_x T_y T_z^{-1} \in [\Torelli_g^1,\Torelli_g^1]_{\bbF_2}$.  As notation, for a group $\Gamma$ and $\lambda \in \Gamma$ let $[\lambda]_2$ be the image
of $\lambda$ in $\HH_1(\Gamma;\bbF_2)$.  Our goal is to prove that the
element $T_x T_y T_z^{-1} \in \Torelli_g^1$ satisfies $[T_x T_y T_z^{-1}]_2 = 0$.

Let $\phi\colon \Mod_g^1 \rightarrow \Mod_g$ be the map that glues a disk
to $x = \partial \Sigma_g^1$ and extends mapping classes over that disk by the identity.
Set $\cD_g^1 = \ker(\phi)$.  The element $T_x T_y T_z^{-1}$ lies in $\cD_g^1$ since $T_x$ becomes trivial
and $y$ and $z$ become isotopic when a disk is glued to $x$:
\Figure{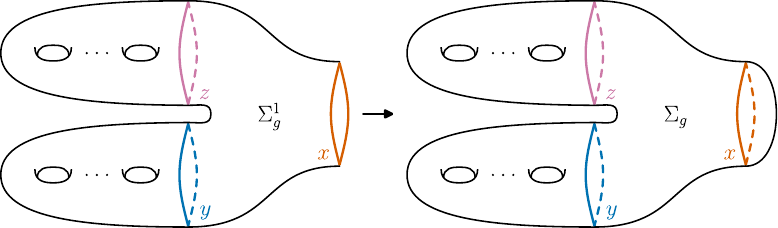}
The group $\cD_g^1$ is called the {\em disk-pushing} subgroup of $\Mod_g^1$.  See
\cite[\S 4.2.5]{FarbMargalitPrimer} for a discussion of it.  Elements
of $\cD_g^1$ ``push'' the glued-on disk around the surface while allowing it to rotate.  Letting
$U\Sigma_g$ be the unit tangent bundle of $\Sigma_g$, since $g \geq 2$ we have
$\cD_g^1 \cong \pi_1(U\Sigma_g)$.  Moreover,
elements of $\cD_g^1$ act trivially on $\HH_1(\Sigma_g^1)$, so $\cD_g^1 \subset \Torelli_g^1$.  
For $\lambda \in \cD_g^1$, let $[\lambda]_2^D$ be the corresponding element of
$\HH_1(\cD_g^1;\bbF_2)$.  From the above, we see that to 
prove that $[T_x T_y T_z^{-1}]_2=0$ in $\HH_1(\Torelli_g^1;\bbF_2)$,
it is enough to prove that $[T_x T_y T_z^{-1}]_2^D=0$.

Let $\Sigma_{g,1}$ be a genus-$g$ surface with one puncture.  Let $\Mod_{g,1}$ 
be the mapping class group of $\Sigma_{g,1}$.  The map
$\phi\colon \Mod_g^1 \rightarrow \Mod_g$ factors as
\[\begin{tikzcd}
\Mod_g^1 \arrow{r}{\phi'} & \Mod_{g,1} \arrow{r}{\phi''} & \Mod_g,
\end{tikzcd}\]
where $\phi'$ glues a punctured disk to $x = \partial \Sigma_g^1$ and extends mapping classes over that punctured 
disk by the identity and $\phi''\colon \Mod_{g,1} \rightarrow \Mod_g$ fills in the puncture.

Set $\cP_{g,1} = \ker(\phi'')$.  The group $\cP_{g,1}$ is called the {\em point-pushing} subgroup of
$\Mod_{g,1}$.  Elements of $\cP_{g,1}$ ``push'' the puncture around the surface, and
$\cP_{g,1} \cong \pi_1(\Sigma_g)$.  The disk-pushing subgroup $\cD_g^1$ and the point-pushing
subgroup $\cP_{g,1}$ fit into the following commutative diagram with exact rows:
\begin{equation}
\label{eqn:unittangentbundleseq}
\begin{tikzcd}[column sep=small, row sep=tiny]
1 \arrow{r} & \Z \arrow{r} \arrow[equals]{d} & \cD_g^1 \arrow{r} \arrow[equals]{d} & \cP_{g,1}       \arrow{r} \arrow[equals]{d} & 1 \\
1 \arrow{r} & \Z \arrow{r}                   & \pi_1(U\Sigma_g)  \arrow{r}         & \pi_1(\Sigma_g) \arrow{r}                   & 1.
\end{tikzcd}
\end{equation}
Here the rows are central extensions, and the central $\Z$ subgroup of $\cD_g^1$ is generated by
$T_x$.  Passing to $\HH_1$ with $\bbF_2$-coefficients, we get an exact sequence
\begin{equation}
\label{eqn:diskpointh1}
\begin{tikzcd}
\bbF_2 \arrow{r} & \HH_1(\cD_g^1;\bbF_2) \arrow{r} & \HH_1(\cP_{g,1};\bbF_2) \arrow{r} & 0,
\end{tikzcd}
\end{equation}
where the image of $\bbF_2$ in $\HH_1(\cD_g^1;\bbF_2)$ is generated by $[T_x]_2^D \in \HH_1(\cD_g^1;\bbF_2)$.

Recall that our goal is to prove that $[T_x T_y T_z^{-1}]_2^D = 0$ in $\HH_1(\cD_g^1;\bbF_2)$.
Let $\oy$ and $\oz$ be the images of $y$ and $z$ in $\Sigma_{g,1}$.  The image
of $T_x T_y T_z^{-1} \in \cD_g^1$ in $\cP_{g,1}$ is thus $T_{\oy} T_{\oz}^{-1}$.  Write
$[T_{\oy} T_{\oz}^{-1}]_2^P$ for its image in $\HH_1(\cP_{g,1};\bbF_2) \cong \HH_1(\Sigma_g;\bbF_2)$.  We claim
that $[T_{\oy} T_{\oz}^{-1}]_2^P=0$.  Indeed, let $\gamma \in \cP_{g,1}$ be the element represented
by the following based separating curve:
\Figure{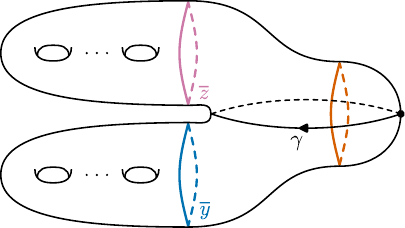}
Since $\oy$ and $\oz$ are the boundary components of a regular neighborhood of this based
curve, the standard formula for the effect of point-pushing along a simple closed curve
(see \cite[Fact 4.7]{FarbMargalitPrimer}) shows that $T_{\oy} T_{\oz}^{-1} = \gamma$.
Separating curves are null-homologous, so this implies that $[T_{\oy} T_{\oz}^{-1}]_2^P = [\gamma]_2^P = 0$,
as claimed.

Since $[T_x T_y T_z^{-1}]_2^D$ lies in the kernel of the map $\HH_1(\cD_g^1;\bbF_2) \rightarrow \HH_1(\cP_{g,1};\bbF_2)$
from \eqref{eqn:diskpointh1}, we conclude that either
$[T_x T_y T_z^{-1}]_2^D = [T_x]_2^D$ or $[T_x T_y T_z^{-1}]_2^D = 0$.
As we will see below, $[T_x]_2^D \neq 0$.  
To prove the lemma, we must therefore prove that $[T_x T_y T_z^{-1}]_2^D \neq [T_x]_2^D$.

Let $\sigma\colon \Torelli_g^1 \rightarrow \BCJ_3(g)$ be the BCJ homomorphism.  Since
the codomain of $\sigma$ is an elementary abelian $2$-group, the restriction of $\sigma$ to
$\cD_g^1$ factors through $\HH_1(\cD_g^1;\bbF_2)$.  It is therefore enough to prove
that $\sigma(T_x T_y T_z^{-1}) \neq \sigma(T_x)$.  Lemma~\ref{lemma:seppantsker} says that
$\sigma(T_x T_y T_z) = 0$, so since $T_z^{-2} \in \ker(\sigma)$ it follows that
$\sigma(T_x T_y T_z^{-1}) = 0$.  However, the formulas in \S \ref{subsection:septwistsbcj}
imply that $\sigma(T_x) \neq 0$.  The lemma follows.
\end{proof}

\section{Basic commutators and the subgroups \texorpdfstring{$\Rel_g^b$}{Ngb} and \texorpdfstring{$\KRel_g^b$}{KRgb}}
\label{section:simplebasic}

Fix $g \geq 2$ and $b \in \{0,1\}$.  In this section, we give an alternate
characterization of the basic commutators and study the subgroups generated
by our purported generators for the kernel of the BCJ homomorphism and the
commutator subgroup of $\Torelli_g^b$.

\subsection{Basic commutators}
Let $\igeom(-,-)$ be the geometric intersection number.  Recall that a basic commutator in
$\Torelli_g^b$ is an element of the form $[T_x,T_y T_z^{-1}] \in \Torelli_g^b$, where:
\begin{itemize}
\item[(a)] $T_x$ is a separating twist; and
\item[(b)] $T_y T_z^{-1}$ is a BP map; and
\item[(c)] $\igeom(x,y)=\igeom(x,z)=2$; and
\item[(d)] putting $y$ and $z$ into minimal position with $x$, both 
$x \cup y$ and $x \cup z$ separate $\Sigma_g^b$ into two subsurfaces.
\end{itemize}

\subsection{Basic commutator criterion}
To clarify the meaning of this, we introduce some definitions.  
Set $H = \HH_1(\Sigma_g^b;\bbF_2)$.
Let $x$ be a separating curve in $\Sigma_g^b$.  Let $S$ and $T$ be the subsurfaces
on either side of $x$.  Let $U \subset H$ be the image of $\HH_1(S;\bbF_2)$ and
let $V \subset H$ be the image of $\HH_1(T;\bbF_2)$.  The subspaces $U$ and $V$ of
$H$ are orthogonal with respect to the algebraic intersection pairing, and $H = U \oplus V$.
We call $H = U \oplus V$ the {\em symplectic splitting} associated to $x$.

Over $\bbF_2$, it makes sense
to talk about the homology class $[\gamma] \in H$ of a simple closed
curve $\gamma$ on $\Sigma_g^b$.  There is no need to choose an orientation
on $\gamma$.  For a BP map $T_y T_z^{-1} \in \Torelli_g^b$, we have
$[y] = [z]$.  The common value $[y] = [z]$ is the {\em $\bbF_2$-homology class}
of the BP map $T_y T_z^{-1}$.  We then have:

\begin{lemma}
\label{lemma:basiccommutators}
Fix $g \geq 2$ and $b \in \{0,1\}$.  Set $H = \HH_1(\Sigma_g^b;\bbF_2)$.
Consider a separating twist $T_x \in \Torelli_g^b$ and a BP map $T_y T_z^{-1} \in \Torelli_g^b$.
Let $H = U \oplus V$ be the symplectic splitting associated to $x$ and let 
$h \in H$ be the $\bbF_2$-homology class of $T_y T_z^{-1}$.  Then 
$[T_x,T_y T_z^{-1}]$ is a basic commutator if and only if the following are
satisfied:
\begin{itemize}
\item[(c$\hspace{2pt}'$)] $\igeom(x,y) \leq 2$ and $\igeom(x,z) \leq 2$; and
\item[(d$\hspace{2pt}'$)] writing $h = u+v$ with $u \in U$ and $v \in V$, we have $u \neq 0$ and $v \neq 0$.
\end{itemize}
\end{lemma}
\begin{proof}
We first introduce some notation.
Let $S$ and $T$ be the subsurfaces on either side of $x$.  
Let $\hSigma$ be the space obtained by collapsing $x$ to a point and
let $\hS$ and $\hT$ be the images of $S$ and $T$ in $\hSigma$.  Both
$\hS$ and $\hT$ are surfaces, and $\hSigma$ is the wedge sum of
$\hS$ and $\hT$:
\Figure{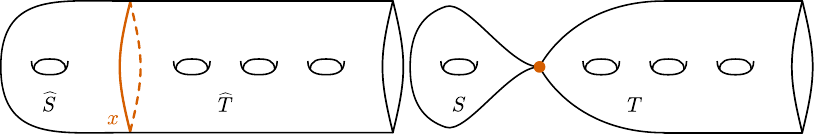}
We can identify $U$ and $V$ with $\HH_1(\hS;\bbF_2)$ and $\HH_1(\hT;\bbF_2)$.
We now divide the proof into two claims:

\begin{unnumberedclaim}
Assume that $[T_x,T_y T_z^{-1}]$ is a basic commutator.  Then (c$\hspace{2pt}'$) and (d$\hspace{2pt}'$) hold.
\end{unnumberedclaim}

Since $[T_x,T_y T_z^{-1}]$ is a basic commutator, we have $\igeom(x,y) =2$ and $\igeom(x,z) = 2$.  It
follows that (c$\hspace{2pt}'$) holds.  Put $y$ into minimal position with $x$, and let
$\hy$ be the image of $y$ in $\hSigma$.  The picture is as follows:
\Figure{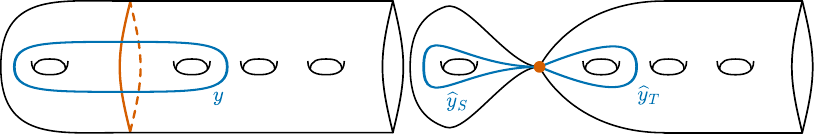}
As this figure shows, $\hy$ is the union of two simple closed curves meeting at a point,
one curve $\hy_S$ in $\hS$ and the other curve $\hy_T$ in $\hT$.  From the definitions,
we have $u = [\hy_S]$ and $v = [\hy_T]$.  Since $[T_x,T_y T_z^{-1}]$ is a basic commutator,
$x \cup y$ separates $\Sigma_g^b$ into two components.  This implies that $y \cap S$ (resp.\ $y \cap T$)
is a nonseparating arc in $S$ (resp.\ $T$), which implies that $\hy_S$ (resp.\ $\hy_T$)
is a nonseparating curve in $\hS$ (resp.\ $\hT$).  
Since nonseparating curves have nonzero $\bbF_2$-homology classes, we conclude that $u \neq 0$
and $v \neq 0$, so (d$\hspace{2pt}'$) holds.

\begin{unnumberedclaim}
Assume that (c$\hspace{2pt}'$) and (d$\hspace{2pt}'$) hold.  Then $[T_x,T_y T_z^{-1}]$ is a basic commutator.
\end{unnumberedclaim}

Put $y$ and $z$ into minimal position with $x$.  We must prove that $\igeom(x,y)=\igeom(x,z)=2$ and that
both $x \cup y$ and $x \cup z$ separate $\Sigma_g^b$ into two components.  The proofs
for $y$ and $z$ are identical, so we will give the details for $y$.

Since $x$ is separating, $\igeom(x,y)$ must be even.  By (c$\hspace{2pt}'$) it
is either $0$ or $2$.  If $\igeom(x,y)=0$, then $y$ lies entirely in either $S$ or $T$.  This implies
that $h = [y]$ lies entirely in $U$ or $V$, contrary to (d$\hspace{2pt}'$).  We conclude that
$\igeom(x,y)=2$.  

Let $\hy$ be the image of $y$ in $\hSigma$.  Since $\igeom(x,y)=2$, this decomposes into
a simple closed curve $\hy_S \subset \hS$ and a simple closed curve $\hy_T \subset \hT$.
The homology classes of $\hy_S$ and $\hy_T$ are $u$ and $v$, so by (d$\hspace{2pt}'$)
they are both nonzero.  Since $\hS$ and $\hT$ are both surfaces with at most one boundary
component, this implies that $\hy_S$ and $\hy_T$ are both nonseparating simple closed curves.
This implies that $x \cup y$ separates $\Sigma_g^b$ into two components, as desired.
\end{proof}

\subsection{The subgroups \texorpdfstring{$\Rel_g^b$}{Ngb} and \texorpdfstring{$\KRel_g^b$}{KRgb}}

Let $\Rel_g^b$ be the subgroup of $\Torelli_g^b$ generated by squares of separating twists,
squares of BP maps, basic commutators, and tri-separating pants maps.  Since each of these
families of elements is closed under conjugation by $\Mod_g^b$, the group $\Rel_g^b$ is
a normal subgroup of $\Mod_g^b$ and hence of $\Torelli_g^b$.  For $g \geq 3$, Theorem~\ref{maintheorem:genker} asserts
that $\Rel_g^b$ is the kernel of the BCJ homomorphism.

Let $\KRel_g^b$ be the subgroup of $\Torelli_g^b$ generated by squares of separating twists,
basic commutators, and tri-separating pants maps.
Each of these families of generators lies in the Johnson kernel $\cK_g^b$; indeed, this
is immediate for squares of separating twists and tri-separating pants maps, and
for basic commutators $[T_x,T_y T_z^{-1}] = T_x (T_y T_z^{-1}) T_x^{-1} (T_y T_z^{-1})^{-1}$
it follows since $T_x \in \cK_g^b$ and $\cK_g^b$ is a normal subgroup of $\Mod_g^b$.
The group $\KRel_g^b$ is also normal in $\Mod_g^b$ and hence in
$\Torelli_g^b$ and $\cK_g^b$.  For $g \geq 3$, Theorem~\ref{maintheorem:gencomm} asserts
that $\KRel_g^b = [\Torelli_g^b,\Torelli_g^b]$.

Just like in our notation $\Mod_g^b$, we will omit the $b$ from $\Rel_g^b$ and $\KRel_g^b$ if $b=0$.
We close this section with the following observation:

\begin{lemma}
\label{lemma:capboundary}
Fix $g \geq 2$.  Let $\phi\colon \Torelli_g^1 \rightarrow \Torelli_g$ be the homomorphism
that glues a disk to $\partial \Sigma_g^1$ and extends mapping classes over that disk
by the identity.  Then $\phi(\Rel_g^1) = \Rel_g$ and $\phi(\KRel_g^1) = \KRel_g$.
\end{lemma}
\begin{proof}
Since the curves that make up the generators for $\Rel_g$ (resp.\ $\KRel_g$) can be isotoped to be disjoint from
the glued-on disk, we have $\Rel_g \subset \phi(\Rel_g^1)$ (resp.\ $\KRel_g \subset \phi(\KRel_g^1)$).  We must
prove the other inclusion.  We do this by studying the images of each family of generators:
\begin{itemize}
\item For a separating twist $T_x \in \Torelli_g^1$, the image $\phi(T_x^2)$ is the square of a separating twist unless
$x$ is isotopic to $\partial \Sigma_g^1$, in which case $\phi(T_x^2) = 1$.  
\item For a BP map $T_y T_z^{-1} \in \Torelli_g^1$, the image $\phi((T_y T_z^{-1})^2)$ is the square of a BP map unless
$y$ and $z$ become isotopic when a disk is glued to $\partial \Sigma_g^1$, in which case
$\phi((T_y T_z^{-1})^2) = 1$.
\item For a tri-separating pants map $T_x T_y T_z \in \Torelli_g^1$, its image $\phi(T_x T_y T_z)$ is
a tri-separating pants map unless one of $\{x,y,z\}$ is isotopic to $\partial \Sigma_g^1$.  Assume
now that one of $\{x,y,z\}$ is isotopic to $\partial \Sigma_g^1$.  Reordering the twists, we can
assume that $z$ is isotopic to $\partial \Sigma_g^1$.  It follows that $x$ and $y$ become isotopic 
when a disk is glued to $\partial \Sigma_g^1$.  Letting $w$ be the separating curve in
$\Sigma_g$ such that $\phi(T_x) = \phi(T_y) = T_w$, we then have $\phi(T_x T_y T_z) = T_w^2$, the
square of a separating twist.
\item For a basic commutator $[T_x, T_y T_z^{-1}] \in \Torelli_g^1$, if $y$ and $z$ become isotopic
when a disk is glued to $\partial \Sigma_g^1$ then $\phi([T_x,T_y T_z^{-1}]) = 1$.  Otherwise, 
Lemma~\ref{lemma:basiccommutators} implies that $\phi([T_x,T_y T_z^{-1}])$ is a
basic commutator.  The point here is that gluing a disk to $\partial \Sigma_g^1$ does
not change the first homology group and can
only cause geometric intersection numbers of curves to decrease.
\end{itemize}
The lemma follows from the above calculations.
\end{proof}

\section{Quotients and a reduction}
\label{section:quotients}

Fix $g \geq 2$ and $b \in \{0,1\}$.
Using the notation from the previous section, for $g \geq 3$
Theorem~\ref{maintheorem:genker} says that $\Rel_g^b$ is the kernel of the BCJ homomorphism
and Theorem~\ref{maintheorem:gencomm} says that $\KRel_g^b = [\Torelli_g^b,\Torelli_g^b]$.  
Define $\QTorelli_g^b = \Torelli_g^b/\Rel_g^b$ and $\QK_g^b = \cK_g^b/\KRel_g^b$.
In this section, we reduce Theorems~\ref{maintheorem:genker} and \ref{maintheorem:gencomm}
to a result about $\QK_g^1$.

\subsection{Johnson homomorphisms and the Johnson kernel}
\label{section:johnsonhomoandker}

We start by recalling some facts from the introduction.
Let $H_{\Z} = \HH_1(\Sigma_g^b) \cong \Z^{2g}$.  Embed $H_{\Z}$ into $\wedge^3 H_{\Z}$
via the map $h \mapsto h \wedge \omega$, where
$\omega \in \wedge^2 H_{\Z} \cong \wedge^2 H_{\Z}^{\ast}$ is the symplectic form.
Johnson (\cite{JohnsonHomo}; see \cite[\S 6.6]{FarbMargalitPrimer} for a textbook reference)
constructed the Johnson homomorphisms, which are surjective homomorphisms of the form
\[\tau \colon \Torelli_g^1 \rightarrow \wedge^3 H_{\Z} \quad \text{and} \quad 
\tau\colon \Torelli_g \rightarrow (\wedge^3 H_{\Z})/H_{\Z}.\]
The Johnson kernel $\cK_g^b$ is the kernel of the Johnson homomorphism on
$\Torelli_g^b$.  We therefore have exact sequences
\[\begin{tikzcd}[column sep=small, row sep=0pt]
1 \arrow{r} & \cK_g^1 \arrow{r} & \Torelli_g^1 \arrow{r}{\tau} & \wedge^3 H_{\Z} \arrow{r} & 1, \\
1 \arrow{r} & \cK_g   \arrow{r} & \Torelli_g   \arrow{r}{\tau} & (\wedge^3 H_{\Z})/H_{\Z}   \arrow{r} & 1.
\end{tikzcd}\]
For $g \geq 2$, Johnson \cite{Johnson2} proved that $\cK_g^b$ is the subgroup of $\Torelli_g^b$ generated
by separating twists.

\subsection{Quotienting the exact sequences}

The following lemma connects the groups $\QK_g^b$ and $\QTorelli_g^b$ to the
exact sequences involving $\cK_g^b$ and $\Torelli_g^b$ discussed in \S \ref{section:johnsonhomoandker}.
In it, we emphasize that we are not asserting that the
maps $\QK_g^1 \rightarrow \QTorelli_g^1$ and $\QK_g \rightarrow \QTorelli_g$
on the second rows of these diagrams are injective.

\begin{lemma}
\label{lemma:naturequotient}
Fix $g \geq 2$.  Set $H_{\Z} = \HH_1(\Sigma_g^1) = \HH_1(\Sigma_g)$ and
$H = \HH_1(\Sigma_g^1;\bbF_2) = \HH_1(\Sigma_g;\bbF_2)$.  We then
have commutative diagrams with exact rows
\[\begin{tikzcd}[column sep=tiny]
1 \arrow{r}
  & \cK_g^1 \arrow[two heads]{d} \arrow{r}
  & \Torelli_g^1 \arrow[two heads]{d} \arrow{r}{\tau}
  & \wedge^3 H_{\Z} \arrow[two heads]{d}{\substack{\text{reduce}\\ \text{mod }2}}
      \arrow{r}
  & 1 \\
  & \QK_g^1 \arrow{r}
  & \QTorelli_g^1 \arrow{r}
  & \wedge^3 H \arrow{r}
  & 1
\end{tikzcd}
\enspace \text{and} \enspace
\begin{tikzcd}[column sep=tiny]
1 \arrow{r}
  & \cK_g \arrow[two heads]{d} \arrow{r}
  & \Torelli_g \arrow[two heads]{d} \arrow{r}{\tau}
  & (\wedge^3 H_{\Z})/H_{\Z} \arrow[two heads]{d}{\substack{\text{reduce}\\ \text{mod }2}}
      \arrow{r}
  & 1 \\
  & \QK_g \arrow{r}
  & \QTorelli_g \arrow{r}
  & (\wedge^3 H)/H \arrow{r}
  & 1
\end{tikzcd}\]
\end{lemma}
\begin{proof}
Fix $b \in \{0,1\}$.  Let $A_g^b$ be the free abelian group 
that is the target of the Johnson homomorphism $\tau$ on $\Torelli_g^b$, so
we have a short exact sequence
\[1 \longrightarrow \cK_g^b \longrightarrow \Torelli_g^b \stackrel{\tau}{\longrightarrow} A_g^b \rightarrow 1.\]
Quotienting by $\KRel_g^b$ and $\Rel_g^b$, we get a commutative diagram of exact sequences
\[\begin{tikzcd}[column sep=tiny, row sep=small]
1 \arrow{r}
  & \cK_g^b \arrow[two heads]{d} \arrow{r}
  & \Torelli_g^b \arrow[two heads]{d} \arrow{r}{\tau}
  & A_g^b \arrow[two heads]{d} \arrow{r}
  & 1 \\
  & \QK_g^b  \arrow{r}
  & \QTorelli_g^b \arrow{r}
  & A_g^b/\tau(\Rel_g^b) \arrow{r}
  & 1. 
\end{tikzcd}\]
It is not clear at this point in the paper if the map
$\QK_g^b \rightarrow \QTorelli_g^b$ is injective, which would require
proving that $\Rel_g^b \cap \cK_g^b = \KRel_g^b$.  Right now we only
know that $\KRel_g^b \subset \Rel_g^b \cap \cK_g^b$.  For $g \geq 3$, we remark that
it will follow from results we prove later in the paper that 
the map $\QK_g^b \rightarrow \QTorelli_g^b$ is in fact injective.
To prove the lemma, we must prove that $\tau(\Rel_g^b) = 2 A_g^b$.

Since $\cK_g^b$ is generated by separating twists it follows that $\cK_g^b$
contains all squares of separating twists and all tri-separating pants maps.  Also, since $A_g^b$ is abelian
and $\cK_g^b = \ker(\tau)$ it follows that $\cK_g^b$ contains all basic commutators.
We deduce that $\tau(\Rel_g^b)$ is the subgroup generated by the image under $\tau$ of the set of squares of
BP maps.  Johnson's calculations in \cite{JohnsonHomo} show that $\tau$ takes the set
of BP maps to a generating set for $A_g^b$, so $\tau$ takes the set of
squares of BP maps to a generating set for $2 A_g^b$.  The lemma follows.
\end{proof}

\subsection{Induced BCJ homomorphisms}

Recall that the BCJ homomorphisms are of the form $\sigma\colon \Torelli_g^1 \rightarrow \BCJ_3(g)$
and $\sigma\colon \Torelli_g \rightarrow \BCJZ_3(g)$.  Since the target of $\sigma$ is an elementary
abelian $2$-group, its kernel contains squares of separating twists, squares of BP maps, and basic commutators.
Moreover, Lemma~\ref{lemma:seppantsker} says that $\ker(\sigma)$ contains tri-separating pants maps.
It follows that the BCJ homomorphisms factor through maps
$\sigma\colon \QTorelli_g^1 \rightarrow \BCJ_3(g)$ and $\sigma\colon \QTorelli_g \rightarrow \BCJZ_3(g)$
that we will call the {\em induced BCJ homomorphisms}.  

Next, recall that Johnson \cite{JohnsonBCJ} proved that $\sigma(\cK_g^1) = \BCJ_2(g)$ and
$\sigma(\cK_g) = \BCJZ_2(g)$.  Just like above, the BCJ homomorphisms on $\cK_g^1$ and
$\cK_g$ factor through maps
$\sigma\colon \QK_g^1 \rightarrow \BCJ_2(g)$ and $\sigma\colon \QK_g \rightarrow \BCJZ_2(g)$
that we will also call the {\em induced BCJ homomorphisms}.  

All of these are related to the exact
sequences in Lemma~\ref{lemma:naturequotient} as follows:

\begin{lemma}
\label{lemma:inducedbcj}
Fix $g \geq 2$.  Set $H = \HH_1(\Sigma_g^1;\bbF_2) = \HH_1(\Sigma_g;\bbF_2)$.  We then
have commutative diagrams with exact rows
\[\begin{tikzcd}[column sep=tiny, row sep=small]
& \QK_g^1 \arrow{r} \arrow{d}{\sigma} & \QTorelli_g^1 \arrow{r} \arrow{d}{\sigma} & \wedge^3 H \arrow{r} \arrow{d}{\cong} & 1 \\
0 \arrow{r} & \BCJ_2(g) \arrow{r}                & \BCJ_3(g) \arrow{r}                      & \BCJ_3(g)/\BCJ_2(g) \arrow{r}             & 0
\end{tikzcd}
\enspace \text{and} \enspace
\begin{tikzcd}[column sep=tiny, row sep=small]
& \QK_g  \arrow{r} \arrow{d}{\sigma} & \QTorelli_g \arrow{r} \arrow{d}{\sigma} & (\wedge^3 H)/H \arrow{r} \arrow{d}{\cong} & 1 \\
0 \arrow{r} & \BCJZ_2(g) \arrow{r}              & \BCJZ_3(g) \arrow{r}                   & \BCJZ_3(g)/\BCJZ_2(g) \arrow{r}               & 0.
\end{tikzcd}\]
In both diagrams, the first two vertical arrows are the induced BCJ homomorphisms.
\end{lemma}
\begin{proof}
Immediate from Lemma~\ref{lemma:naturequotient} and Johnson's work in \cite{JohnsonBCJ} (see the commutative diagrams
in \S \ref{section:torelliabelianization}).
\end{proof}

\subsection{Alternate main theorem}

Parts~\ref{part:2} and \ref{part:3} of this paper are devoted to proving the
following theorem.  Note that in its statement we require $g \geq 3$.

\begin{primedtheorem}{maintheorem:genker}
\label{maintheorem:qkg}
For $g \geq 3$, the induced BCJ homomorphism 
$\sigma\colon \QK_g^1 \rightarrow \BCJ_2(g)$ is an isomorphism.
\end{primedtheorem}

The rest of this section shows how to use Theorem~\ref{maintheorem:qkg} to prove Theorems~\ref{maintheorem:genker}
and \ref{maintheorem:gencomm}.

\subsection{Capping the boundary}

The first step is to prove that Theorem~\ref{maintheorem:qkg} implies a similar
theorem for closed surfaces:

\begin{lemma}
\label{lemma:qkgclosed}
Fix $g \geq 3$.  Assume that the induced BCJ homomorphism $\sigma\colon \QK_g^1 \rightarrow \BCJ_2(g)$
is an isomorphism.  Then the induced
BCJ homomorphism $\sigma\colon \QK_g \rightarrow \BCJZ_2(g)$ is also an isomorphism.
\end{lemma}
\begin{proof}
Let $c\colon \cK_g^1 \rightarrow \cK_g$ be the surjective homomorphism that glues a disk
to $\partial \Sigma_g^1$ and extends mapping classes over it by the identity.  
Lemma~\ref{lemma:capboundary} says that $c(\KRel_g^1) = \KRel_g$.  It follows that
the map $c$ induces a surjective map $\QK_g^1 \rightarrow \QK_g$.
This fits into a commutative diagram
\[\begin{tikzcd}
\QK_g^1 \arrow{r}{\sigma}[swap]{\cong} \arrow[two heads]{d} & \BCJ_2(g) \arrow[two heads]{d} \\
\QK_g   \arrow{r}{\sigma}                                   & \BCJZ_2(g).
\end{tikzcd}\]
To prove the lemma, we must therefore prove that the elements of $\QK_g^1$ corresponding
to elements of the kernel of the projection $\BCJ_2(g) \twoheadrightarrow \BCJZ_2(g)$ lie in the kernel 
of the projection $\QK_g^1 \twoheadrightarrow \QK_g$.  Fix a symplectic basis
$\{a_1,b_1,\ldots,a_g,b_g\}$ for $H = \HH_1(\Sigma_g^1;\bbF_2)$.  Lemma~\ref{lemma:bcj2difference}
says that the kernel
of the projection $\BCJ_2(g) \twoheadrightarrow \BCJZ_2(g)$ is isomorphic to $\bbF_2$
and is generated by the element
\[\kappa = \oa_1 \ob_1 + \cdots + \oa_g \ob_g.\]
Letting $x = \partial \Sigma_g^1$, the formulas in \S \ref{subsection:septwistsbcj} imply
that $\sigma(T_x) = \kappa$.  The lemma therefore follows from the fact that $T_x$ is in
the kernel of the map $\cK_g^1 \rightarrow \cK_g$.
\end{proof} 

\subsection{The proof of Theorem~\ref{maintheorem:genker}}
\label{section:theoremaproof}

We next show how to use Theorem~\ref{maintheorem:qkg} to prove Theorem~\ref{maintheorem:genker}, whose
statement we recall:

\newtheorem*{maintheorem:genker}{Theorem~\ref{maintheorem:genker}}
\begin{maintheorem:genker}
For $g \geq 3$ and $b \in \{0,1\}$, the kernel of the BCJ homomorphism
on $\Torelli_g^b$ is generated by squares of separating twists, squares of BP maps, basic
commutators, and tri-separating pants maps.
\end{maintheorem:genker}
\begin{proof}[Proof, assuming Theorem~\ref{maintheorem:qkg}]
It is enough to prove that the induced BCJ homomorphisms
$\sigma\colon \QTorelli_g^1 \rightarrow \BCJ_3(g)$ and $\sigma\colon \QTorelli_g \rightarrow \BCJZ_3(g)$ are isomorphisms.
Letting $H = \HH_1(\Sigma_g^1;\bbF_2)$, Lemma~\ref{lemma:inducedbcj} gives a commutative diagram with exact rows
\[\begin{tikzcd}[column sep=small]
& \QK_g^1 \arrow{r} \arrow{d}{\sigma} & \QTorelli_g^1 \arrow{r} \arrow{d}{\sigma} & \wedge^3 H \arrow{r} \arrow{d}{\cong} & 1 \\
0 \arrow{r} & \BCJ_2(g) \arrow{r}                & \BCJ_3(g) \arrow{r}                      & \BCJ_3(g)/\BCJ_2(g) \arrow{r}             & 0.
\end{tikzcd}\]
Theorem~\ref{maintheorem:qkg} implies that $\sigma\colon \QK_g^1 \rightarrow \BCJ_2(g)$ is an isomorphism, which
implies that the map $\QK_g^1 \rightarrow \QTorelli_g^1$ in the above diagram is injective.
By the five lemma, we conclude that $\sigma\colon \QTorelli_g^1 \rightarrow \BCJ_3(g)$ is an isomorphism, as
desired.  The proof that $\sigma\colon \QTorelli_g \rightarrow \BCJZ_3(g)$ is an isomorphism
is similar, using Lemma~\ref{lemma:qkgclosed} together with Theorem~\ref{maintheorem:qkg} to 
see that $\sigma\colon \QK_g \rightarrow \BCJZ_2(g)$ is an isomorphism.
\end{proof}

\subsection{The proof of Theorem~\ref{maintheorem:gencomm}}
\label{section:theorembproof}

We close this part of the paper by showing 
how to use Theorem~\ref{maintheorem:qkg} to prove Theorem~\ref{maintheorem:gencomm}, whose
statement we recall:

\newtheorem*{maintheorem:gencomm}{Theorem~\ref{maintheorem:gencomm}}
\begin{maintheorem:gencomm}
For $g \geq 3$ and $b \in \{0,1\}$, the group $[\Torelli_g^b,\Torelli_g^b]$
is generated by squares of separating twists, basic
commutators, and tri-separating pants maps.
\end{maintheorem:gencomm}
\begin{proof}[Proof, assuming Theorem~\ref{maintheorem:qkg}]
Johnson \cite{Johnson3} proved that $[\Torelli_g^b,\Torelli_g^b]$ is the kernel 
of the restriction of the BCJ homomorphism to $\cK_g^b$.  In the introduction,
we described an alternate proof of this using Theorem~\ref{maintheorem:genker}.
It follows that to prove the theorem, it is enough to prove that the indicated
generating set generates the kernel of the restriction of the BCJ homomorphism
to $\cK_g^b$.  The indicated generating set generates $\KRel_g^b$, and
Theorem~\ref{maintheorem:qkg} together with Lemma~\ref{lemma:qkgclosed} imply
that the BCJ homomorphism on $\cK_g^b$ induces an isomorphism from $\QK_g^b = \cK_g^b/\KRel_g^b$
to its target.  The theorem follows.
\end{proof}

\part{The nature of the quotient}
\label{part:2}

It remains to prove Theorem~\ref{maintheorem:qkg}.  This part of the paper
studies properties of $\QK_g^1$ and reduces Theorem~\ref{maintheorem:qkg}
to Theorem~\ref{maintheorem:qkgpres}, which is a description of
$\BCJ_2(g)$ by generators and relations.
We start in \S \ref{section:qkabelian}
by proving that $\QK_g^1$ is an elementary abelian $2$-group equipped with an action of
$\Sp_{2g}(\Z)$.  We then prove in \S \ref{section:qkmod2} that the action of
$\Sp_{2g}(\Z)$ on $\QK_g^1$ factors through $\Sp_{2g}(\bbF_2)$.  Using this,
we reduce Theorem~\ref{maintheorem:qkg} to Theorem~\ref{maintheorem:qkgpres} in
\S \ref{section:qkgenrel}.

\begin{remark}
Everything we prove for the rest of the paper concerns surfaces with one boundary
component, so for instance we will henceforth talk about $\QK_g^1$ but not $\QK_g$.
\end{remark}

\begin{genusassumption}
Throughout Part~\ref{part:2}, we will assume that the genus $g$ satisfies $g \geq 3$.
\end{genusassumption}

\section{The group \texorpdfstring{$\QK_g^1$}{QKg1} is an elementary abelian 2-group}
\label{section:qkabelian}

Fix $g \geq 3$.  In this section, we prove that $\QK_g^1$ is an elementary abelian $2$-group.

\subsection{Generation by genus-1 separating twists}

This requires two lemmas.
A {\em genus-$h$ separating twist} in $\Mod_g^1$ is a separating twist $T_x$ such that
$x$ separates $\Sigma_g^1$ into subsurfaces $S \cong \Sigma_h^1$ and $S' \cong \Sigma_{g-h}^2$:
\Figure{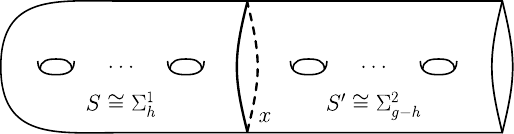}
For $\lambda \in \cK_g^1$, let $[\lambda]_Q$ be its image in $\QK_g^1$.  Our first
lemma is as follows:

\begin{lemma}
\label{lemma:qkgen}
For $g \geq 3$, the group $\QK_g^1$ is generated by the set of all $[T_x]_Q$ such
that $T_x$ is a genus-$1$ separating twist in $\Mod_g^1$.
\end{lemma}
\begin{proof}
Johnson \cite{Johnson2} proved that $\cK_g^1$ is generated by separating twists.  Letting
$T_x$ be a genus-$h$ separating twist with $h \geq 2$, it is enough to prove that $[T_x]_Q$ can be written
as a product of terms of the form $[T_{x'}]_Q$ with $T_{x'}$ a genus-$1$ separating twist.  
We can find a tri-separating pants map $T_x T_y T_z$ such that
$T_y$ is a genus-$1$ separating twist and $T_z$ is a genus-$(h-1)$ separating twist.  See the figure
here:
\Figure{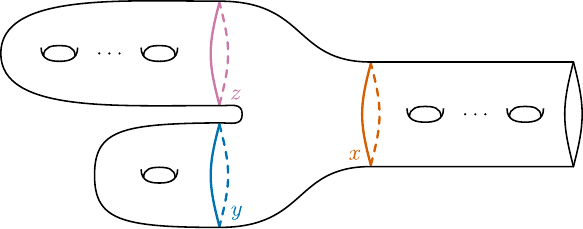}
We have $\QK_g^1 = \cK_g^1/\KRel_g^1$ and $\KRel_g^1$ contains all tri-separating pants maps.
It follows that $[T_x T_y T_z]_Q = 1$, so since $\KRel_g^1$ also contains all squares of separating twists
we have
\[[T_x]_Q = [T_z]_Q^{-1} [T_y]_Q^{-1} = [T_z]_Q [T_y]_Q.\]
The mapping class $T_y$ is a genus-$1$ separating twist, and
by induction $[T_z]_Q$ can be written as a product of terms of the form
$[T_{z'}]_Q$ with $T_{z'}$ a genus-$1$ separating twist.  The
lemma follows.
\end{proof}

\subsection{Trivial Torelli action}

The conjugation action of $\Mod_g^1$ on its normal subgroup $\cK_g^1$
descends to an action of $\Mod_g^1$ on $\QK_g^1$.  We next prove
that the restriction of this action to $\Torelli_g^1$ is trivial:

\begin{lemma}
\label{lemma:torellitrivial}
For $g \geq 3$, the action of $\Torelli_g^1$ on $\QK_g^1$ is trivial.
\end{lemma}
\begin{proof}
Let $T_x$ be a genus-$1$ separating twist.  By Lemma~\ref{lemma:qkgen} it
is enough to prove that $\Torelli_g^1$ acts trivially on $[T_x]_Q$.  Let
$\Gamma < \Torelli_g^1$ be the stabilizer of $[T_x]_Q$, i.e., the subgroup
of all $f \in \Torelli_g^1$ such that $[f T_x f^{-1}]_Q = [T_x]_Q$.  Our goal is to prove
that $\Gamma = \Torelli_g^1$.  To do this, we will prove that $\Gamma$ contains
a generating set for $\Torelli_g^1$.

We first enumerate some elements of $\Gamma$:
\begin{itemize}
\item If $f \in \Torelli_g^1$ satisfies $f(x) = x$, then 
\[[f T_x f^{-1}]_Q = [T_{f(x)}]_Q = [T_x]_Q\]
and thus $f \in \Gamma$.  In particular, $\Gamma$ contains all
BP maps $T_y T_z^{-1}$ such that $y$ and $z$ are both disjoint
from $x$.
\item Now assume that $T_y T_{z}^{-1}$ is a BP map such that
$[T_x,T_y T_z^{-1}]$ is a basic commutator.  
We have $\QK_g^1 = \cK_g^1/\KRel_g^1$, where $\KRel_g^1$ contains all
basic commutators.  It follows that
\[[T_x (T_y T_z^{-1}) T_x^{-1} (T_y T_z^{-1})^{-1}]_Q = 1,\]
so $[(T_y T_z^{-1}) T_x (T_y T_z^{-1})^{-1}]_Q = [T_x]_Q$.  We deduce that $T_y T_{z}^{-1} \in \Gamma$.
\end{itemize}
Summarizing, the subgroup $\Gamma$ of $\Torelli_g^1$ contains the set $S$ of all 
BP maps $T_y T_z^{-1}$ such that either $y \cup z$ is disjoint from $x$ or
$[T_x,T_y T_z^{-1}]$ is a basic commutator.  It follows from work 
of Johnson \cite{Johnson1} that $S$ generates $\Torelli_g^1$, so $\Gamma = \Torelli_g^1$.  In fact,
what Johnson proved is as follows.  Draw $x$ as follows:
\Figure{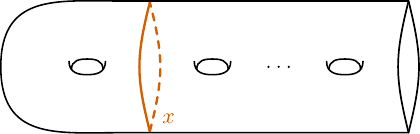}
Johnson \cite{Johnson1} constructed a finite generating set $S'$ for $\Torelli_g^1$ consisting
of BP maps $T_y T_z^{-1}$ that intersect $x$ as described above, so $S' \subset S$.  The lemma follows.
\end{proof}

\subsection{Abelian group}

We now use the above results to prove:

\begin{proposition}
\label{proposition:qkabelian}
Fix $g \geq 3$.  The following hold:
\begin{itemize}
\item The group $\QK_g^1$ is an elementary abelian $2$-group.
\item The action of $\Mod_g^1$ on $\QK_g^1$ factors through an action
of $\Sp_{2g}(\Z)$.
\end{itemize}
\end{proposition}
\begin{proof}
That the action of $\Mod_g^1$ on $\QK_g^1$ factors through an action of
$\Mod_g^1 / \Torelli_g^1 = \Sp_{2g}(\Z)$
is immediate from Lemma~\ref{lemma:torellitrivial}.  That lemma also 
implies that the subgroup $\cK_g^1$ of $\Torelli_g^1$ acts trivially on $\QK_g^1$.  For $a,b \in \QK_g^1$,
this implies that $a b a^{-1} = b$, so $ab = ba$.  We conclude that $\QK_g^1$
is abelian.  We will henceforth write elements of it using additive notation.

Lemma~\ref{lemma:qkgen} says that $\QK_g^1$ is generated
by elements of the form $[T_x]_Q$, where $T_x$ is a genus-$1$ separating twist.
We have $\QK_g^1 = \cK_g^1/\KRel_g^1$, where $\KRel_g^1$ contains all
squares of separating twists.  For a separating
twist $T_x$, this implies that $2 [T_x]_Q=0$.  We conclude that 
$\QK_g^1$ is an elementary abelian $2$-group, as desired. 
\end{proof}

\section{The group \texorpdfstring{$\QK_g^1$}{QKg1} is a representation of \texorpdfstring{$\Sp_{2g}(\bbF_2)$}{Sp(2g,Z/2)}}
\label{section:qkmod2}

Fix $g \geq 3$.
Our next goal is to prove that the action of $\Sp_{2g}(\Z)$ on $\QK_g^1$ factors through $\Sp_{2g}(\bbF_2)$.
As notation, let $H_{\Z} = \HH_1(\Sigma_g^1)$ and let $\omega(-,-)$ be the $\Z$-valued algebraic intersection
form on $H_{\Z}$.  We call it $\omega$ to distinguish it from $\hiota(-,-)$, which is the
$\bbF_2$-valued algebraic intersection form on $H = \HH_1(\Sigma_g^1;\bbF_2)$.

\subsection{Symplectic summands and separating twists}

A {\em symplectic summand} of $H_{\Z}$ is a direct summand $V_{\Z}$ of $H_{\Z}$ such that the restriction
of $\omega$ to $V_{\Z}$ is symplectic, i.e., such that $\omega$ identifies $V_{\Z}$ with its dual
$V_{\Z}^{\ast} = \Hom(V_{\Z},\Z)$.  Equivalently, letting $\perp$ be the orthogonal complement with
respect to $\omega$ we have $H_{\Z} = V_{\Z} \oplus V_{\Z}^{\perp}$.  We remark that if $V_{\Z}$ is a subgroup
of $H_{\Z}$ such that the restriction of $\omega$ to $V_{\Z}$ is symplectic, then $V_{\Z}$ is automatically
a direct summand and $H_{\Z} = V_{\Z} \oplus V_{\Z}^{\perp}$. 

One way these arise is as follows.
Let $T_x$ be a genus-$h$ separating twist in $\Torelli_g^1$.  The curve $x$ separates $\Sigma_g^1$ into
a subsurface $S$ with $S \cong \Sigma_h^1$ and a subsurface $S'$ with $S' \cong \Sigma_{g-h}^2$.  It
is immediate that $\HH_1(S)$ is a symplectic summand of $H_{\Z}$; indeed, letting $B < \HH_1(S')$ be
the subgroup spanned by the boundary components, we have a decomposition
$H_{\Z} = \HH_1(S) \oplus (\HH_1(S')/B)$
that is orthogonal with respect to $\omega$.  The reason we
quotient by $B$ here is that the map $\HH_1(S') \rightarrow \HH_1(\Sigma_g^1)$ is not injective.
We will call $\HH_1(S)$ the symplectic summand
{\em associated} to $T_x$.

\subsection{Symplectic summand generators}

Let $V_{\Z}$ be a nonzero symplectic summand of $H_{\Z}$.  Johnson \cite{JohnsonConj} proved that there
is a separating twist $T_x$ such that $V_{\Z}$ is the symplectic summand associated to $T_x$.  Johnson \cite{JohnsonConj}
also proved that if $T_{x'}$ is another separating twist such that $V_{\Z}$ is the symplectic summand associated
to $T_{x'}$, then there exists some $f \in \Torelli_g^1$ such that $f T_x f^{-1}= T_{x'}$.  Since
$\Torelli_g^1$ acts trivially on $\QK_g^1$ (see Lemma~\ref{lemma:torellitrivial}), we have
\[[T_x]_Q = [f T_x f^{-1}]_Q = [T_{f(x)}]_Q = [T_{x'}]_Q.\]
It follows that the element $[T_x]_Q \in \QK_g^1$ only depends on $V_{\Z}$.  We will denote it
by $[V_{\Z}]_Q$.  We have $V_{\Z} \cong \Z^{2h}$ for some $h \geq 1$, and we call $h$ the {\em genus}
of $V_{\Z}$.  The action of $\Sp_{2g}(\Z)$ on $\QK_g^1$ given by Proposition~\ref{proposition:qkabelian} 
is compatible with this notation: for $M \in \Sp_{2g}(\Z)$, we have $M \cdot [V_{\Z}]_Q = [M(V_{\Z})]_Q$. 

\begin{remark}
In the above argument, we allow $V_{\Z} = H_{\Z}$.  Indeed, letting $\partial$ be the
boundary component of $\Sigma_g^1$ we have $[H_{\Z}]_Q = [T_{\partial}]_Q$.
\end{remark}

The following is immediate from Lemma~\ref{lemma:qkgen} and the above discussion:

\begin{lemma}
\label{lemma:qkgenz}
For $g \geq 3$, the group $\QK_g^1$ is generated by the set of all $[V_{\Z}]_Q$ such
that $V_{\Z}$ is a genus-$1$ symplectic summand of $H_{\Z} = \HH_1(\Sigma_g^1)$.
\end{lemma}

\subsection{Relations between symplectic summand generators}
Two symplectic summands $V_{\Z},W_{\Z} < H_{\Z}$
are {\em orthogonal} if $\omega(v,w) = 0$ for all $v \in V_{\Z}$ and $w \in W_{\Z}$.  This implies that
$V_{\Z} \oplus W_{\Z}$ is a symplectic summand of $H_{\Z}$.  
The elements $[V_{\Z}]_Q$ and $[W_{\Z}]_Q$ of $\QK_g^1$ satisfy the following relation:

\begin{lemma}
\label{lemma:orthogonalrelation}
Fix $g \geq 3$.  Set $H_{\Z} = \HH_1(\Sigma_g^1)$, and let $V_{\Z},W_{\Z} < H_{\Z}$ be nonzero orthogonal
symplectic summands.  We then have $[V_{\Z}]_Q+[W_{\Z}]_Q = [V_{\Z} \oplus W_{\Z}]_Q$.
\end{lemma}
\begin{proof}
The same argument Johnson used in \cite{JohnsonConj} to prove that there is a separating twist
whose associated symplectic summand is a given symplectic summand shows more generally
that there exist disjoint simple closed separating curves $y$ and $z$ such that:
\begin{itemize}
\item $V_{\Z}$ is the symplectic summand associated to $T_y$; and 
\item $W_{\Z}$ is the symplectic summand associated to $T_z$.
\end{itemize}  
We can then find a simple closed separating curve $x$ that is disjoint from $y$ and $z$ such that
$x \cup y \cup z$ bounds a pair of pants.  See the following figure:
\Figure{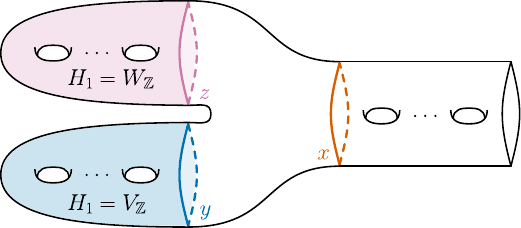}
As this figure shows, the symplectic summand associated to $T_x$ is $V_{\Z} \oplus W_{\Z}$.  
In $\QK_g^1 = \cK_g^1/\KRel_g^1$, the group $\KRel_g^1$ contains all tri-separating pants maps.
Since $T_x T_y T_z$ is a tri-separating pants map, it follows that
\[0 = [T_x T_y T_z]_Q = [T_x]_Q + [T_y]_Q + [T_z]_Q = [V_{\Z} \oplus W_{\Z}]_Q + [V_{\Z}]_Q + [W_{\Z}]_Q.\]
Since $\QK_g^1$ is an elementary abelian $2$-group, this implies that $[V_{\Z}]_Q+[W_{\Z}]_Q = [V_{\Z} \oplus W_{\Z}]_Q$,
as desired.
\end{proof}

\subsection{Completing symplectic bases}

Before proving our final result in this section, we prove
a technical lemma:

\begin{lemma}
\label{lemma:completespbasis}
Fix $g \geq 3$.  Set $H_{\Z} = \HH_1(\Sigma_g^1)$, and let $\omega(-,-)$ be the
algebraic intersection pairing on $H_{\Z}$.  Let $V_{\Z}$ be a genus-$h$ symplectic summand
of $H_{\Z}$.  For some $k \geq 0$, let $a_1,\ldots,a_k \in V_{\Z}$ be elements whose
span $\Span{a_1,\ldots,a_k}$ is a rank-$k$ direct summand of $V_{\Z}$ and which
satisfy $\omega(a_i,a_j)=0$ for all $1 \leq i,j \leq k$.  We can then complete
$\{a_1,\ldots,a_k\}$ to a symplectic basis $\{a_1,b_1,\ldots,a_h,b_h\}$ for $V_{\Z}$.
\end{lemma}
\begin{proof}
The proof will be by induction on $k$.  The base case $k=0$ just asserts that
every symplectic summand has a symplectic basis, which is standard.  Assume now
that $k \geq 1$ and that the lemma is true whenever $k$ is smaller.  

Since $\Span{a_1,\ldots,a_k}$ is a rank-$k$ direct summand of $V_{\Z}$, there is
a homomorphism $\lambda\colon V_{\Z} \rightarrow \Z$ such that
$\lambda(a_1)=1$ and $\lambda(a_i)=0$ for $2 \leq i \leq k$.  Since the
restriction of $\omega$ to $V_{\Z}$ is symplectic, there exists $b_1 \in V_{\Z}$
such that $\omega(v,b_1) = \lambda(v)$ for all $v \in V_{\Z}$.  In particular,
$\omega(a_1,b_1) = 1$ and $\omega(a_i,b_1) = 0$ for $2 \leq i \leq k$.  

Let $V'_{\Z}$ be the orthogonal complement of $\Span{a_1,b_1}$ in $V_{\Z}$.  This
is a symplectic summand and $a_2,\ldots,a_k \in V'_{\Z}$.  In free abelian groups, direct
summands contained in a subgroup are direct summands of that subgroup, so $\Span{a_2,\ldots,a_k}$ is a direct
summand of $V'_{\Z}$.  By induction, we can
complete $\{a_2,\ldots,a_k\}$ to a symplectic basis $\{a_2,b_2,\ldots,a_h,b_h\}$
for $V'_{\Z}$.  It follows that $\{a_1,b_1,\ldots,a_h,b_h\}$ is a symplectic
basis for $V_{\Z}$.
\end{proof}

\subsection{Factoring through \texorpdfstring{$\boldsymbol{\Sp_{2g}(\bbF_2)}$}{Sp(2g,F2)}}

The following result is the main goal of this section.
The calculation underlying its proof first appeared
in van den Berg's thesis \cite{VanDenBergThesis};
see especially the proof of \cite[Lemma 3.5.3]{VanDenBergThesis}.

\begin{proposition}
\label{proposition:factor2}
Fix $g \geq 3$.  The action of $\Sp_{2g}(\Z)$ on $\QK_g^1$ factors
through $\Sp_{2g}(\bbF_2)$.
\end{proposition}
\begin{proof}
Set $H_{\Z} = \HH_1(\Sigma_g^1)$, and let $\omega(-,-)$ be the algebraic intersection
pairing on $H_{\Z}$.  For $v \in H_{\Z}$, let $\tau_v \in \Sp_{2g}(\Z)$
be the associated symplectic transvection, so
\[\tau_v(h) = h + \omega(v,h) v \quad \text{for all $h \in H_{\Z}$}.\]
The group $\Sp_{2g}(\Z)$ is generated by the set of symplectic transvections
$\tau_v$ such that $v \in H_{\Z}$ is primitive, that is, not divisible by
any $d \geq 2$ (see, e.g., the discussion in \cite[\S 6.3.3]{FarbMargalitPrimer}).  
We will call these the {\em primitive transvections}.
We start with the following relation in $\QK_g^1$.  Recall that $\perp$ indicates
the orthogonal complement in $H_{\Z}$ with respect to $\omega$.

\begin{unnumberedclaim}
Let $V_{\Z}$ be a symplectic summand of $H_{\Z}$ of genus at least $2$.  Let $a_1,a_2 \in V_{\Z}$
be elements such that $\omega(a_1,a_2) = 0$ and such that their span
$\Span{a_1,a_2}$ is a rank-$2$ direct summand of $V_{\Z}$.  Also, let $w \in V_{\Z}^{\perp}$.
Then for all $n \in \Z$ we have
\[[\tau_{n a_1 + a_2 + w}(V_{\Z})]_Q + [\tau_{n a_1+w}(V_{\Z})]_Q + [\tau_{a_2+w}(V_{\Z})]_Q + [V_{\Z}]_Q = 0.\]
\end{unnumberedclaim}
\begin{proof}[Proof of claim]
Just like in the statement of the claim, for elements $x_1,\ldots,x_k \in H_{\Z}$ and 
subsets $J_1,\ldots,J_{\ell} \subset H_{\Z}$ we will
write $\Span{x_1,\ldots,x_k,J_1,\ldots,J_{\ell}}$ for the subgroup of $H_{\Z}$ spanned by the $x_i$ and the 
union of the $J_j$.
Using Lemma~\ref{lemma:completespbasis}, complete
the $a_i$ to a symplectic basis $\{a_1,b_1,\ldots,a_h,b_h\}$ for $V_{\Z}$.  Set
$I = \{a_3,b_3,\ldots,a_h,b_h\}$.  Using the relation from Lemma~\ref{lemma:orthogonalrelation},
\begin{align}
\label{claim:factor2.1}
[\tau_{n a_1+a_2 + w}(V_{\Z})]_Q 
                              &= [\Span{a_1,b_1+n(n a_1+a_2 + w),a_2,b_2+(n a_1+a_2 + w),I}]_Q \\
                              &= [\Span{a_1,b_1+nw,a_2,b_2+w,I}]_Q \notag\\
                              &= [\Span{a_1,b_1+nw}]_Q + [\Span{a_2,b_2+w,I}]_Q.\notag
\end{align}
Similarly, we have\noeqref{claim:factor2.2}\noeqref{claim:factor2.3}
\begin{align}
\label{claim:factor2.2}
[\tau_{n a_1+w}(V_{\Z})]_Q &= [\Span{a_1,b_1+nw}]_Q + [\Span{a_2,b_2,I}]_Q, \\
\label{claim:factor2.3}
[\tau_{a_2+w}(V_{\Z})]_Q &= [\Span{a_1,b_1}]_Q + [\Span{a_2,b_2+w,I}]_Q, \\
\label{claim:factor2.4}
[V_{\Z}]_Q &= [\Span{a_1,b_1}]_Q + [\Span{a_2,b_2,I}]_Q.
\end{align}
Adding \eqref{claim:factor2.1}-\eqref{claim:factor2.4} and using the fact that $\QK_g^1$ is an elementary abelian $2$-group, we 
deduce that
\begin{align*}
&[\tau_{n a_1+a_2 + w}(V_{\Z})]_Q + [\tau_{n a_1+w}(V_{\Z})]_Q + [\tau_{a_2+w}(V_{\Z})]_Q + [V_{\Z}]_Q \\
&=2 [\Span{a_1,b_1+nw}]_Q + 2[\Span{a_2,b_2+w,I}]_Q + 2[\Span{a_1,b_1}]_Q + 2[\Span{a_2,b_2,I}]_Q = 0.\qedhere
\end{align*}
\end{proof}

We now turn to the proof of the proposition.  We must prove that the kernel of the mod-$2$ reduction
map $\Sp_{2g}(\Z) \rightarrow \Sp_{2g}(\bbF_2)$ acts trivially on $\QK_g^1$.  This kernel
is called the {\em level-$2$ congruence subgroup} of $\Sp_{2g}(\Z)$, and is denoted
$\Sp_{2g}(\Z,2)$.  It is generated by squares of primitive transvections;
see \cite[\S 10]{MennickeCongruenceSp}.\footnote{Fixing a symplectic basis $\{a_1,b_1,\ldots,a_g,b_g\}$ 
for $H_{\Z}$, the
main theorem of \cite[\S 10]{MennickeCongruenceSp} says that for all $\ell \geq 2$ the
level-$\ell$ congruence subgroup $\Sp_{2g}(\Z,\ell) = \ker(\Sp_{2g}(\Z) \rightarrow \Sp_{2g}(\Z/\ell))$
is the smallest normal subgroup containing $\tau_{a_1}^{\ell}$.  For
$\phi \in \Sp_{2g}(\Z)$ we have $\phi \tau_{a_1}^{\ell} \phi^{-1} = \tau_{\phi(a_1)}^{\ell}$.  Each
$\phi(a_1)$ is primitive, so the smallest normal subgroup containing $\tau_{a_1}^{\ell}$ is generated
by a set of $\ell$-th powers of primitive transvections.  We conclude that
$\Sp_{2g}(\Z,\ell)$ is generated by $\ell$-th powers of primitive transvections.}  
Fixing a primitive $v \in H_{\Z}$, we must therefore prove that $\tau_{v}^2$ acts trivially on $\QK_g^1$.

Lemma~\ref{lemma:qkgenz} says that $\QK_g^1$ is generated by $[V_{\Z}]_Q$ with 
$V_{\Z}$ a genus-$1$ symplectic summand of $H_{\Z}$.  It is enough to prove that
$\tau_v^2$ acts trivially on such a $[V_{\Z}]_Q$.  We will prove that $\tau_v^2$
acts trivially on $[V_{\Z}]_Q$ 
for arbitrary nonzero symplectic summands $V_{\Z}$.  Below we will prove this
for $V_{\Z}$ of genus at least $2$.
To see that this implies
the general case, assume that $V_{\Z}$ has genus $1$.  
Using the relation from Lemma~\ref{lemma:orthogonalrelation},
we have
$[V_{\Z}]_Q + [V_{\Z}^{\perp}]_Q = [H_{\Z}]_Q$.
Since $g \geq 3$, both $V_{\Z}^{\perp}$ and $H_{\Z}$ have genus at least $2$.  If $\tau_{v}^2$
fixes $[V_{\Z}^{\perp}]_Q$ and $[H_{\Z}]_Q$, it will also fix $[V_{\Z}]_Q$.

We have now reduced to the case where $V_{\Z}$ has genus at least $2$.
We must prove that $\tau^2_v$ acts trivially on $[V_{\Z}]_Q$.
Write $v = n a_1 + w$ with $a_1 \in V_{\Z}$ primitive and $n \in \Z$ and $w \in V_{\Z}^{\perp}$.
Using Lemma~\ref{lemma:completespbasis}, complete $\{a_1\}$ to a symplectic basis $\{a_1,b_1,\ldots,a_h,b_h\}$ for $V_{\Z}$.  By
assumption we have $h \geq 2$, so it makes sense to discuss $a_2$.  Observe first that
\begin{align}
\label{eqn:reducetotwox}
[\tau^2_v(V_{\Z})]_Q = [\tau^2_{n a_1+w}(V_{\Z})]_Q &= [\Span{a_1,b_1+2n(na_1+w),a_2,b_2,\ldots,a_h,b_h}]_Q \\
                             &= [\Span{a_1,b_1+2n(2na_1+w),a_2,b_2,\ldots,a_h,b_h}]_Q \notag\\
                             &= [\tau_{2n a_1+w}(V_{\Z})]_Q.\notag
\end{align}
We will therefore focus on $[\tau_{2n a_1+w}(V_{\Z})]_Q$.  In the following calculation,
we will use the above claim in the first, third, and fourth equality.  The second equality
rewrites the term we are considering in a clever way, and the final equality cancels
terms using the fact that we are working in an elementary abelian $2$-group:
\begin{align*}
&[\tau_{2n a_1+w}(V_{\Z})]_Q + [\tau_{a_2+w}(V_{\Z})]_Q + [V_{\Z}]_Q \\
&=[\tau_{2n a_1+a_2+w}(V_{\Z})]_Q \\
&=[\tau_{n a_1 + (n a_1 + a_2) + w}(V_{\Z})]_Q \\
&=[\tau_{n a_1+w}(V_{\Z})]_Q + [\tau_{n a_1 + a_2 + w}(V_{\Z})]_Q + [V_{\Z}]_Q \\
&=[\tau_{n a_1+w}(V_{\Z})]_Q + [\tau_{n a_1 + w}(V_{\Z})]_Q + [\tau_{a_2 + w}(V_{\Z})]_Q + [V_{\Z}]_Q + [V_{\Z}]_Q \\
&=[\tau_{a_2 + w}(V_{\Z})]_Q.
\end{align*}
Subtracting $[\tau_{a_2 + w}(V_{\Z})]_Q$ from both sides of this, we get that
$[\tau_{2n a_1+w}(V_{\Z})]_Q + [V_{\Z}]_Q = 0$.  By \eqref{eqn:reducetotwox},
we conclude that $[\tau^2_v(V_{\Z})]_Q = [V_{\Z}]_Q$, as desired.
\end{proof}

\section{Reduction to generators and relations for \texorpdfstring{$\BCJ_2(g)$}{B(2g,2)}}
\label{section:qkgenrel}

Fix $g \geq 3$.  This section reduces Theorem~\ref{maintheorem:qkg} to 
Theorem~\ref{maintheorem:qkgpres}, which gives a description of $\BCJ_2(g)$ 
by generators and relations.

\subsection{Symplectic subspaces and orthogonality}

Recall that $H = \HH_1(\Sigma_g^1;\bbF_2)$.  This
is equipped with an $\bbF_2$-valued symplectic form $\hiota(-,-)$.  We will use $\hiota(-,-)$
to define orthogonal complements and orthogonality for subspaces of $H$.
We have previously defined symplectic summands for $H_{\Z} = \HH_1(\Sigma_g^1)$, and
we now introduce similar terminology for $H$.

A {\em symplectic subspace} of $H$ is a subspace $V < H$ on which the
symplectic form restricts to a symplectic form.  Since $H \cong \bbF_2^{2g}$ is a vector
space, a symplectic subspace $V$ of $H$ is automatically a direct summand
and satisfies $H = V \oplus V^{\perp}$.  We have $V \cong \bbF_2^{2h}$ for some
$h \leq g$, and we call $h$ the {\em genus} of $V$.

\subsection{Lifting symplectic subspaces}

We will need the following lemma:

\begin{lemma}
\label{lemma:liftspsubspaces}
Fix $g \geq 3$.  Let $H_{\Z} = \HH_1(\Sigma_g^1)$ and $H = \HH_1(\Sigma_g^1;\bbF_2)$.
Let $V(1),\ldots,V(k)$ be pairwise orthogonal nonzero symplectic subspaces of $H$.  We can
find pairwise orthogonal symplectic summands $V_{\Z}(1),\ldots,V_{\Z}(k)$ of $H_{\Z}$
such that each $V_{\Z}(i)$ projects to $V(i)$ under the projection $H_{\Z} \rightarrow H$.
\end{lemma}
\begin{proof}
Without loss of generality, we can assume that $V(1) \oplus \cdots \oplus V(k) = H$; indeed,
if the $V(i)$ span a proper subspace of $H$ we can add the symplectic subspace
$\left(V(1) \oplus \cdots \oplus V(k)\right)^{\perp}$.  For $1 \leq i \leq k$, let $h(i)$
be the genus of $V(i)$ and let $\{a(i)_1,b(i)_1,\ldots,a(i)_{h(i)},b(i)_{h(i)}\}$
be a symplectic basis for $V(i)$.  The set 
\[\Set{$a(i)_j,b(i)_j$}{$1 \leq i \leq k$, $1 \leq j \leq h(i)$}\]
is a symplectic basis for $H$.  Since the map $\Sp_{2g}(\Z) \rightarrow \Sp_{2g}(\bbF_2)$
is surjective, we can lift this to a symplectic basis for $H_{\Z}$.  In this lifted
symplectic basis, let $a_{\Z}(i)_j$ and $b_{\Z}(i)_j$ be the elements projecting
to $a(i)_j$ and $b(i)_j$, respectively.  For $1 \leq i \leq k$, let
$V_{\Z}(i)$ be the span of $\Set{$a_{\Z}(i)_j,b_{\Z}(i)_j$}{$1 \leq j \leq h(i)$}$.
The $V_{\Z}(i)$ are pairwise orthogonal symplectic summands of $H_{\Z}$ such that
each $V_{\Z}(i)$ projects to $V(i)$.
\end{proof}

\subsection{Action of level-2 subgroup}

Recall from the proof of Proposition~\ref{proposition:factor2} that $\Sp_{2g}(\Z,2)$
denotes the level-$2$ subgroup of $\Sp_{2g}(\Z)$, i.e., the kernel of the map
$\Sp_{2g}(\Z) \rightarrow \Sp_{2g}(\bbF_2)$.  We will need the following:

\begin{lemma}
\label{lemma:level2action}
Fix $g \geq 3$.  Let $H_{\Z} = \HH_1(\Sigma_g^1)$ and $H = \HH_1(\Sigma_g^1;\bbF_2)$.
Let $V$ be a nonzero symplectic subspace of $H$ and let $V_{\Z}$ and $V'_{\Z}$ be
symplectic summands of $H_{\Z}$ that both project to $V$ under the projection
$H_{\Z} \rightarrow H$.  There exists $f \in \Sp_{2g}(\Z,2)$ such that
$f(V_{\Z}) = V'_{\Z}$.
\end{lemma}
\begin{proof}
Choose $\phi \in \Sp_{2g}(\Z)$ such that $\phi(V_{\Z}) = V'_{\Z}$.  We will
multiply $\phi$ by an appropriate matrix to make it lie in $\Sp_{2g}(\Z,2)$.
Let $\ophi \in \Sp_{2g}(\bbF_2)$ be the image of $\phi$.  The matrix $\ophi$
preserves the ordered decomposition $H = V \oplus V^{\perp}$.  Let
\begin{align*}
j\colon \Sp(V) \times \Sp(V^{\perp}) &\rightarrow \Sp_{2g}(\bbF_2),\\
j_{\Z}\colon \Sp(V_{\Z}) \times \Sp(V_{\Z}^{\perp}) &\rightarrow \Sp_{2g}(\Z)
\end{align*} 
be the natural inclusions.
We can write $\ophi = j(\opsi_1 \times \opsi_2)$ for some $\opsi_1 \in \Sp(V)$ and $\opsi_2 \in \Sp(V^{\perp})$.
Since the map from the integral symplectic group to the $\bbF_2$-symplectic group is surjective, we
can find $\psi_1 \in \Sp(V_{\Z})$ and $\psi_2 \in \Sp(V_{\Z}^{\perp})$ that map to $\psi_1$ and $\psi_2$,
respectively. 
Set $\psi = j_{\Z}(\psi_1 \times \psi_2)$.  Since $\psi(V_{\Z}) = V_{\Z}$, the element
$f = \phi \circ \psi^{-1} \in \Sp_{2g}(\Z)$ takes $V_{\Z}$ to $V_{\Z}'$ and maps to the identity in $\Sp_{2g}(\bbF_2)$.
In other words, $f \in \Sp_{2g}(\Z,2)$, as desired.
\end{proof}

\subsection{Mod-2 symplectic subspace generators and relations}

Let $V$ be a nonzero symplectic subspace of $H$.  By Lemma~\ref{lemma:liftspsubspaces}, we can
lift $V$ to a symplectic
summand $V_{\Z}$ of $H_{\Z}$ that projects to $V$ under the map $H_{\Z} \rightarrow H$.
This gives us an element $[V_{\Z}]_Q$ of $\QK_g^1$.  If $V'_{\Z}$ is another symplectic
summand of $H_{\Z}$ that projects to $V$, then Lemma \ref{lemma:level2action} says that there exists some
$f \in \Sp_{2g}(\Z,2)$ such that $f(V_{\Z}) = V'_{\Z}$.  Since the action of $\Sp_{2g}(\Z)$ on $\QK_g^1$
factors through $\Sp_{2g}(\bbF_2)$ (Proposition~\ref{proposition:factor2}), we have
\[[V_{\Z}]_Q = f \cdot [V_{\Z}]_Q = [f(V_{\Z})]_Q = [V'_{\Z}]_Q.\]
In other words, the element $[V_{\Z}]_Q$ of $\QK_g^1$ only depends on $V$.
We will denote it by $[V]_Q$.  We have the following:

\begin{lemma}
\label{lemma:qkgenrel2}
Let $g \geq 3$ and let $H = \HH_1(\Sigma_g^1;\bbF_2)$.  The following hold:
\begin{itemize}
\item[(i)] The group $\QK_g^1$ is generated by elements of the form
$[V]_Q$ with $V$ a genus-$1$ symplectic subspace of $H$.
\item[(ii)] Let $V,W < H$ be nonzero orthogonal symplectic subspaces.  
We then have $[V]_Q+[W]_Q = [V \oplus W]_Q$.
\end{itemize}
\end{lemma}
\begin{proof}
In light of the above discussion, these follow from Lemmas~\ref{lemma:qkgenz} and \ref{lemma:orthogonalrelation}, which
are the corresponding results for symplectic summands of $H_{\Z} = \HH_1(\Sigma_g^1)$.  The only fact
that needs to be added is that in (ii) we can lift $V$ and $W$ to orthogonal symplectic summands
$V_{\Z}$ and $W_{\Z}$ of $H_{\Z}$, which follows from Lemma~\ref{lemma:liftspsubspaces}.  
\end{proof}

\subsection{Vector space given by generators and relations}
\label{section:liftedbcj}

Recall that $H = \HH_1(\Sigma_g^1;\bbF_2)$.
Define $\Pres_g^1$ to be the $\bbF_2$-vector space given by generators and relations as follows:
\begin{itemize}
\item The generators of $\Pres_g^1$ are formal symbols $\SymSub{V}$ with $V < H$ a nonzero
symplectic subspace.
\item The relations of $\Pres_g^1$ are as follows.  Let $V,W < H$ be nonzero orthogonal symplectic subspaces.
We then have the relation $\SymSub{V} + \SymSub{W} = \SymSub{V \oplus W}$.
\end{itemize}
For later use, note that while we have been assuming that $g \geq 3$ the vector space $\Pres_g^1$ 
also makes sense for $g=1$ and $g=2$.  Returning to our standing assumption that $g \geq 3$,
Lemma~\ref{lemma:qkgenrel2} implies that there is a map $\rho\colon \Pres_g^1 \rightarrow \QK_g^1$ defined on generators as follows:
\[\rho(\SymSub{V}) = [V]_Q.\]
It will follow from our results below that this is an isomorphism.

Recall that the induced BCJ homomorphism is of the form
$\sigma\colon \QK_g^1 \rightarrow \BCJ_2(g)$.  We will recall the form
of its target $\BCJ_2(g)$ later when we start doing explicit calculations.
We will call the map $\beta = \sigma \circ \rho$ from $\Pres_g^1$ to $\BCJ_2(g)$
the {\em lifted BCJ homomorphism}.  

\subsection{Alternate main theorem, II}

Part~\ref{part:3} of this paper is devoted to proving the
following theorem:

\begin{primeprimedtheorem}{maintheorem:genker}
\label{maintheorem:qkgpres}
For $g \geq 3$, the lifted BCJ homomorphism $\beta\colon \Pres_g^1 \rightarrow \BCJ_2(g)$ is
an isomorphism.
\end{primeprimedtheorem}

As we will see, this theorem will quickly imply Theorem~\ref{maintheorem:qkg}.

\subsection{Proof of \texorpdfstring{Theorem~\ref{maintheorem:qkg}}{Theorem A'}}
\label{section:theoremaprimeproof}

We recall the statement of Theorem~\ref{maintheorem:qkg}:

\newtheorem*{maintheorem:qkg}{Theorem~\ref{maintheorem:qkg}}
\begin{maintheorem:qkg}
For $g \geq 3$, the induced BCJ homomorphism
$\sigma\colon \QK_g^1 \rightarrow \BCJ_2(g)$ is an isomorphism.
\end{maintheorem:qkg}
\begin{proof}[Proof, assuming Theorem~\ref{maintheorem:qkgpres}]
Let $\rho\colon \Pres_g^1 \rightarrow \QK_g^1$ be the map defined above.
Lemma~\ref{lemma:qkgenrel2} implies that $\rho$ is surjective.  By definition, the lifted
BCJ homomorphism $\beta\colon \Pres_g^1 \rightarrow \BCJ_2(g)$ equals the composition
\[\begin{tikzcd}
\Pres_g^1 \arrow[two heads]{r}{\rho} & \QK_g^1 \arrow{r}{\sigma} & \BCJ_2(g).
\end{tikzcd}\]
Theorem~\ref{maintheorem:qkgpres} says that this is an isomorphism.  Since $\rho$ is surjective,
this implies that $\sigma$ is an isomorphism, as desired.
\end{proof}

\part{Generators and relations}
\label{part:3}

This final part of the paper proves Theorem~\ref{maintheorem:qkgpres}.  This will complete the
proofs of our main theorems. 
Indeed, we proved at the end of Part~\ref{part:2} that Theorem~\ref{maintheorem:qkgpres} implies
Theorem~\ref{maintheorem:qkg}, and we also proved at the end of Part~\ref{part:1}
that Theorem~\ref{maintheorem:qkg} implies Theorems~\ref{maintheorem:genker} and \ref{maintheorem:gencomm}.

There are four sections.  In \S \ref{section:part3intro}, we recall some
details about the BCJ homomorphism and its target, and also explain how
to generalize the statement of Theorem~\ref{maintheorem:qkgpres} from $g \geq 3$
to $g \geq 2$.  Next, in \S \ref{section:qkgpresgenus23}
we prove Theorem~\ref{maintheorem:qkgpres} in the special cases $g=2$ and $g=3$.  These
two cases will form the base of our inductive proof of the general case.  We then
describe an inductive presentation of the target of the BCJ homomorphism in
\S \ref{section:inductivebcj}.  Finally, we prove Theorem~\ref{maintheorem:qkgpres} in
\S \ref{section:qkgpresproof}.

\begin{genusassumption}
Throughout Part~\ref{part:3}, we will assume that the genus $g$ satisfies $g \geq 2$.
We will need this assumption rather than $g \geq 3$ for our inductive arguments.
\end{genusassumption}

\section{Preliminaries on the lifted BCJ homomorphism}
\label{section:part3intro}

Fix $g \geq 2$.
For $g \geq 3$, the lifted BCJ homomorphism is of the form $\beta\colon \Pres_g^1 \rightarrow \BCJ_2(g)$.  This section recalls
the structure of $\BCJ_2(g)$ and establishes a few preliminary results, and also
shows how to generalize the lifted BCJ homomorphism to the case $g=2$.  Let $H = \HH_1(\Sigma_g^1;\bbF_2)$
and let $\hiota(-,-)$ be the $\bbF_2$-valued intersection form on $H$.

\subsection{Reminder of the target of the lifted BCJ homomorphism}

We start by reminding the reader about the definition of $\BCJ_2(g)$.
Let $\bbF_2[\oH]$ be the ring of polynomials in the formal variables $\Set{$\oh$}{$h \in H$}$ and let $\BCJ(g)$
be the quotient of $\bbF_2[\oH]$ by the ideal generated by:
\begin{itemize}
\item $f^2-f$ for $f \in \bbF_2[\oH]$; and
\item $\overline{h_1+h_2} - (\oh_1 + \oh_2 + \hiota(h_1,h_2))$ for all $h_1,h_2 \in H$.
\end{itemize}
Let $\BCJ_{n}(g)$ be the image in $\BCJ(g)$ of the subspace of $f \in \bbF_2[\oH]$ of degree at most $n$,
so we have an increasing chain of vector spaces
\[0 = \BCJ_{-1}(g) \subsetneq \BCJ_0(g) \subsetneq \BCJ_1(g) \subsetneq \cdots \subsetneq \BCJ_{2g}(g) = \BCJ(g).\]
For $0 \leq n \leq 2g$, the dimension of $\BCJ_{n}(g)$ is
\[\binom{2g}{0} + \binom{2g}{1} + \cdots + \binom{2g}{n}.\]
In particular, $\BCJ_2(g)$ has dimension
\[\binom{2g}{0} + \binom{2g}{1} + \binom{2g}{2} = 1 + 2g + \frac{2g(2g-1)}{2} = 2g^2+g+1.\]

\subsection{Calculating the lifted BCJ homomorphism}

Recall that $\Pres_g^1$ is the $\bbF_2$-vector space given by generators and relations as follows:
\begin{itemize}
\item The generators of $\Pres_g^1$ are formal symbols $\SymSub{V}$ with $V < H$ a nonzero
symplectic subspace.
\item The relations of $\Pres_g^1$ are as follows.  Let $V,W < H$ be nonzero orthogonal symplectic subspaces.
We then have the relation $\SymSub{V} + \SymSub{W} = \SymSub{V \oplus W}$.
\end{itemize}
As we observed when we first defined it, this makes sense for all $g \geq 1$.  We have:

\begin{lemma}
\label{lemma:liftedbcjcalc}
Fix $g \geq 3$, and let $\beta\colon \Pres_g^1 \rightarrow \BCJ_2(g)$ be
the lifted BCJ homomorphism.  Let $V$ be a nonzero symplectic subspace
of $H = \HH_1(\Sigma_g^1;\bbF_2)$ and let $\{x_1,y_1,\ldots,x_h,y_h\}$ be a symplectic
basis for $V$.  Then $\beta(\SymSub{V}) = \ox_1 \oy_1 + \cdots + \ox_h \oy_h$.
\end{lemma}
\begin{proof}
By definition $\beta(\SymSub{V})$ is the image under the BCJ homomorphism of
a separating twist $T_x$ such that the symplectic summand associated to $x$
projects to $V$.  In light of this, the lemma is immediate from the
formula for the image of $T_x$ under the BCJ homomorphism from \S \ref{subsection:septwistsbcj}.
\end{proof}

\subsection{Extending to genus two}

Up until now, the lifted BCJ homomorphism $\beta\colon \Pres_g^1 \rightarrow \BCJ_2(g)$
has only been defined for $g \geq 3$.  Extend it to $g=2$
by defining $\beta\colon \Pres_2^1 \rightarrow \BCJ_2(2)$ using the formula
from Lemma \ref{lemma:liftedbcjcalc}.  This requires checking two things:
\begin{itemize}
\item For a nonzero symplectic subspace $V$ of $H$ with symplectic basis
$\{x_1,y_1,\ldots,x_h,y_h\}$, the formula $\beta(\SymSub{V}) = \ox_1 \oy_1 + \cdots + \ox_h \oy_h$
does not depend on the choice of symplectic basis.  This can be verified directly, or
alternatively one can just observe that the formula for the ordinary BCJ homomorphism
from \S \ref{subsection:septwistsbcj} also works for $g=2$.  For that formula to make
sense, $\ox_1 \oy_1 + \cdots + \ox_h \oy_h$ cannot depend on the choice
of symplectic basis.
\item Using that formula on generators, the relations in $\Pres_2^1$ go to
relations in $\BCJ_2(2)$.  This is immediate.
\end{itemize}
We will henceforth use this extension and talk about the lifted BCJ homomorphism $\beta\colon \Pres_g^1 \rightarrow \BCJ_2(g)$
for $g \geq 2$.  In our arguments below, the only thing we use
about the lifted BCJ homomorphism is the formula from Lemma \ref{lemma:liftedbcjcalc}.

\subsection{Surjectivity}

We close this section with the following observation:

\begin{lemma}
\label{lemma:betasurjective}
For $g \geq 2$, the lifted BCJ homomorphism $\beta\colon \Pres_g^1 \rightarrow \BCJ_2(g)$ is surjective.
\end{lemma}
\begin{proof}
Johnson \cite{JohnsonBCJ} proved that the ordinary BCJ homomorphism $\sigma\colon \cK_g^1 \rightarrow \BCJ_2(g)$
is surjective for $g \geq 2$.  His proof of this only uses the formula for the image under
$\sigma$ of a separating twist.  The lemma follows.
\end{proof}

\section{Theorem \texorpdfstring{\ref{maintheorem:qkgpres}}{A''} when \texorpdfstring{$g=2$}{g=2} and \texorpdfstring{$g=3$}{g=3}}
\label{section:qkgpresgenus23}

Recall that Theorem~\ref{maintheorem:qkgpres} says that the lifted BCJ homomorphism
$\beta\colon \Pres_g^1 \rightarrow \BCJ_2(g)$ is an isomorphism for $g \geq 3$.  For
the sake of our induction, we will actually prove it for $g \geq 2$.  This section establishes
it for the base cases $g=2$ and $g=3$.

\subsection{Genus two}

We start with $g=2$:

\begin{lemma}
\label{lemma:qkgpresg2}
The lifted BCJ homomorphism $\beta\colon \Pres_2^1 \rightarrow \BCJ_2(2)$ is an isomorphism.
\end{lemma}
\begin{proof}
The dimension of $\BCJ_2(2)$ is
\[\binom{4}{0} + \binom{4}{1} + \binom{4}{2} = 1 + 4 + 6 = 11.\]
Since $\beta$ is surjective (Lemma~\ref{lemma:betasurjective}), it is enough to prove that
$\Pres_2^1$ has dimension at most $11$.  The $\bbF_2$-vector space $\Pres_2^1$ has two kinds of
generators:
\begin{itemize}
\item Letting $H = \HH_1(\Sigma_2^1;\bbF_2)$, a single generator $\SymSub{H}$.
\item For each genus-$1$ symplectic subspace $V \cong \bbF^2$ of $H$, a generator $\SymSub{V}$.  The
group $\Sp_{4}(\bbF_2)$ acts transitively on such symplectic subspaces, and the stabilizer
of $V$ is isomorphic to $\Sp_2(\bbF_2) \times \Sp_2(\bbF_2)$ since the stabilizer preserves the decomposition
$H = V \oplus V^{\perp}$.  It follows that
there are\footnote{Here we are using the standard fact
that for $g \geq 1$ and $p \geq 2$ a prime we have $|\Sp_{2g}(\bbF_p)| = (p^{2g}-1)p^{2g-1}(p^{2g-2}-1)p^{2g-3} \cdots (p^2-1)p$.}
\[\frac{|\Sp_{4}(\bbF_2)|}{|\Sp_2(\bbF_2)|^2} = \frac{(2^4-1)2^3(2^2-1)2}{((2^2-1)2)^2} 
= \frac{15 \times 8 \times 3 \times 2}{6^2} = \frac{720}{36} = 20\]
such generators.
\end{itemize}
For each genus-$1$ symplectic subspace $V$ of $H$, we have a single relation
\[\SymSub{V} + \SymSub{V^{\perp}} = \SymSub{H}.\]
These are all the relations in $\Pres_2^1$.  They allow us to eliminate exactly one of $\SymSub{V}$ or $\SymSub{V^{\perp}}$ for each
genus-$1$ symplectic subspace $V$.  After doing this, half of the 20 generators $\SymSub{V}$ survive.  We conclude that
the dimension of $\Pres_2^1$ is at most $10 + 1 = 11$, as desired.
\end{proof}

\subsection{Genus three}

We next handle the case $g=3$:

\begin{lemma}
\label{lemma:qkgpresg3}
The lifted BCJ homomorphism $\beta\colon \Pres_3^1 \rightarrow \BCJ_2(3)$ is an isomorphism.
\end{lemma}
\begin{proof}
The dimension of $\BCJ_2(3)$ is
\[\binom{6}{0} + \binom{6}{1} + \binom{6}{2} = 1 + 6 + 15 = 22.\]
Since $\beta$ is surjective (Lemma~\ref{lemma:betasurjective}), it is enough to prove that
$\Pres_3^1$ is at most $22$-dimensional.  As we will explain below, we verified this on a computer.
The $\bbF_2$-vector space $\Pres_3^1$ has three kinds of generators:
\begin{itemize}
\item Letting $H = \HH_1(\Sigma_3^1;\bbF_2)$, a single generator $\SymSub{H}$.
\item For each genus-$1$ symplectic subspace $V \cong \bbF^2$ of $H$, a generator $\SymSub{V}$.  The
group $\Sp_{6}(\bbF_2)$ acts transitively on such symplectic subspaces, and the stabilizer
of $V$ is isomorphic to $\Sp_2(\bbF_2) \times \Sp_4(\bbF_2)$ since the stabilizer preserves the decomposition
$H = V \oplus V^{\perp}$.  It follows that
there are
\begin{align*}
\frac{|\Sp_{6}(\bbF_2)|}{|\Sp_2(\bbF_2)| \times |\Sp_4(\bbF_2)|} &= \frac{(2^6-1)2^5(2^4-1)2^3(2^2-1)2}{(2^2-1)2 \times (2^4-1)2^3(2^2-1)2} \\
&= \frac{63 \times 32 \times 15 \times 8 \times 3 \times 2}{3 \times 2 \times 15 \times 8 \times 3 \times 2} = \frac{1451520}{4320} = 336
\end{align*}
such generators. 
\item For each genus-$2$ symplectic subspace $W \cong \bbF^4$ of $H$, a generator $\SymSub{W}$.  Since
$W^{\perp}$ is a genus-$1$ symplectic subspace, these are in bijection with the genus-$1$ symplectic subspaces.
We therefore conclude that there are $336$ such generators.
\end{itemize}
Letting $\{a_1,b_1,a_2,b_2,a_3,b_3\}$ be a symplectic basis for $H$, we have
\[\beta(\SymSub{H}) = \oa_1 \ob_1 + \oa_2 \ob_2 + \oa_3 \ob_3 \neq 0.\]
It follows that $\SymSub{H}$ is nonzero and spans a $1$-dimensional subspace
of $\Pres_3^1$.  Let $\Lambda$ be the quotient of $\Pres_3^1$ by the subspace spanned by $\SymSub{H}$.
It is enough to prove that $\Lambda$ is at most $21$-dimensional.  For $x \in \Pres_3^1$, let
$\ux$ be the image of $x$ in $\Lambda$.

For a genus-$2$ symplectic subspace $W$ of $H$, we have a relation
\[\uSymSub{W} + \uSymSub{W^{\perp}} = \uSymSub{W \oplus W^{\perp}} = \uSymSub{H} = 0\]
in $\Lambda$.  Since $W^{\perp}$ is a genus-$1$ symplectic subspace of $H$, this allows us to eliminate all the generators
$\uSymSub{W}$.  In other words, $\Lambda$ is spanned by the generators $\uSymSub{V}$ as $V$ ranges over the 336 genus-$1$
symplectic subspaces of $H$.

For an orthogonal decomposition $H = V_1 \oplus V_2 \oplus V_3$ with each $V_i$ a genus-$1$ symplectic subspace of $H$, we have
a relation
\[\uSymSub{V_1} + \uSymSub{V_2} + \uSymSub{V_3} = \uSymSub{V_1 \oplus V_2} + \uSymSub{V_3} = \uSymSub{V_1 \oplus V_2 \oplus V_3} = \uSymSub{H} = 0\]
in $\Lambda$.  Let $\Lambda'$ be the $\bbF_2$-vector space with the following presentation:
\begin{itemize}
\item The generators of $\Lambda'$ are formal symbols $\LGen{V}$ with $V$ a genus-$1$
symplectic subspace of $H$.
\item The relations of $\Lambda'$ are as follows.  For each orthogonal decomposition $H = V_1 \oplus V_2 \oplus V_3$ with the $V_i$ genus-$1$ symplectic subspaces of $H$,
we have a relation $\LGen{V_1} + \LGen{V_2} + \LGen{V_3} = 0$.
\end{itemize}
We have a surjection $\Lambda' \rightarrow \Lambda$ taking a generator $\LGen{V}$ to $\uSymSub{V}$, so it is enough
to prove that $\Lambda'$ has dimension at most $21$.

As we observed above, there are 336 generators $\LGen{V}$ of $\Lambda'$.  The group $\Sp_{6}(\bbF_2)$ acts transitively
on (unordered) orthogonal decompositions $H = V_1 \oplus V_2 \oplus V_3$ with each $V_i$ a genus-$1$ symplectic subspace of $H$.
The stabilizer of one such decomposition is isomorphic to an extension of $\Sp_2(\bbF_2) \times \Sp_2(\bbF_2) \times \Sp_2(\bbF_2)$
by the symmetric group $\fS_3$ on $3$ letters.\footnote{The symmetric group $\fS_3$ appears since we are considering
unordered decompositions.  Reordering the $V_i$ does not change the relation.}  It follows that the following counts the number of
relations in $\Lambda'$:
\begin{align*}
\frac{|\Sp_{6}(\bbF_2)|}{|\Sp_2(\bbF_2)|^3 \times |\fS_3|} &= \frac{(2^6-1)2^5(2^4-1)2^3(2^2-1)2}{((2^2-1)2)^3 \times 6} \\
&= \frac{63 \times 32 \times 15 \times 8 \times 3 \times 2}{6^3 \times 6} = \frac{1451520}{1296} = 1120.
\end{align*}
Using software written in C++ by the authors and available on their websites, we enumerated these 336 generators and 1120
relations as follows:
\begin{enumerate}
\item Each two-dimensional subspace of $H \cong \bbF_2^6$ can be uniquely represented as the row 
space of a $2 \times 6$ matrix $M$ over $\bbF_2$ that is in reduced row echelon form.  In other words,
for some $1 \leq p < q \leq 6$ we have:
\begin{itemize}
\item The first nonzero entry of row $1$ appears in position $p$ and the first nonzero entry of row $2$ appears in position $q$; and
\item The entry in position $q$ of row $1$ is $0$.
\end{itemize}
Here are some examples of matrices with this shape:
\[
\left(\begin{matrix} 1 & 0 & \ast & \ast & \ast & \ast \\
                     0 & 1 & \ast & \ast & \ast & \ast \end{matrix}\right)
\quad \text{and} \quad
\left(\begin{matrix} 0 & 1 & \ast & 0 & \ast & \ast \\
                       0 & 0 & 0    & 1 & \ast & \ast \end{matrix}\right).
\]
Such a matrix represents a genus-$1$ symplectic subspace if and only if the algebraic intersection pairing between its rows is $1 \in \bbF_2$.
\item The program first enumerates all possible matrices as in (a).  For each, it checks whether the
algebraic intersection pairing between its rows is $1$.  If so, it adds it to the list of generators.
\item For each $1 \leq i < j < k \leq 336$, the program then checks to see if the $i^{\text{th}}$ and $j^{\text{th}}$ and $k^{\text{th}}$ generators
represent pairwise orthogonal genus-$1$ symplectic subspaces.  If so, it adds a relation to the list of relations.
\end{enumerate}
The program outputs a $336 \times 1120$ matrix with entries in $\bbF_2$.  Each column has exactly three nonzero entries.  It then calculates
the dimension of the column space of this matrix, which turns out to be $315$.  It follows that the dimension of $\Lambda'$ is $336-315=21$, as desired.
\end{proof}

\section{Inductive presentation of \texorpdfstring{$\BCJ_2(g)$}{BCJ2}}
\label{section:inductivebcj}

Let $g \geq 2$, and set $H = \HH_1(\Sigma_g^1;\bbF_2)$.
This section introduces an inductive description of $\BCJ_2(g)$.  Our proof of the general case
of Theorem~\ref{maintheorem:qkgpres} in the next section proceeds by showing that
$\Pres_g^1$ has a similar inductive description for $g \geq 4$.

\subsection{Notation for a generalization of \texorpdfstring{$\BCJ_n(g)$}{BCJn}}

Let $U$ be an $\bbF_2$-vector space equipped with a symplectic form $\hiota(-,-)$; for instance,
$U$ might be a symplectic subspace of $H$.  Generalizing our construction of $\BCJ(g)$,
we define $\BCJ(U)$ to be the quotient of the ring $\bbF_2[\oU]$
of polynomials in the formal variables $\Set{$\ou$}{$u \in U$}$
by the ideal generated by:
\begin{itemize}
\item $f^2-f$ for $f \in \bbF_2[\oU]$; and
\item $\overline{u_1+u_2} - (\ou_1 + \ou_2 + \hiota(u_1,u_2))$ for all $u_1,u_2 \in U$.
\end{itemize}
Letting $h$ be the genus of $U$, this has a filtration
\[0 = \BCJ_{-1}(U) \subsetneq \BCJ_0(U) \subsetneq \BCJ_1(U) \subsetneq \cdots \subsetneq \BCJ_{2h}(U) = \BCJ(U),\]
where $\BCJ_n(U)$ is the image in $\BCJ(U)$ of the set of polynomials
in $\bbF_2[\oU]$ of degree at most $n$.

This construction is functorial in the following sense.  Let $U$ and $U'$ be $\bbF_2$-vector
spaces equipped with symplectic forms and let $\iota\colon U \rightarrow U'$ be
a linear map preserving the symplectic form.  This implies that $\iota$ is injective.
The map $\iota\colon U \rightarrow U'$ induces
a map $\BCJ_n(U) \rightarrow \BCJ_n(U')$ that is also easily seen to be injective.
We will use this functoriality frequently when $U$ is a symplectic subspace of $U'$,
in which case the map $\iota\colon U \rightarrow U'$ is the inclusion.

\subsection{Standard symplectic basis and deleted subspaces}
\label{section:standardspdecomp}

Fix a symplectic basis $\{a_1,b_1,\ldots,a_g,b_g\}$ for $H$, which we will call
the {\em standard symplectic basis}.  For $1 \leq i \leq g$, let
$H(i) = \Span{a_i,b_i}$.  This is a genus-$1$ symplectic subspace of $H$, and
we have an orthogonal decomposition
\[H = H(1) \oplus H(2) \oplus \cdots \oplus H(g).\]
For distinct $1 \leq i_1,\ldots,i_k \leq g$, let $H(\hsi_1,\ldots,\hi_k)$
be the subspace of $H$ obtained by deleting the terms $H(i_j)$ from the
above decomposition, so
\[H(\hsi_1,\ldots,\hi_k) = \bigoplus_{\substack{1 \leq i \leq g \\ i \neq i_1,\ldots,i_k}} H(i).\]
Finally, define 
\[\BCJ_n(g;\hi_1,\ldots,\hi_k) = \BCJ_n(H(\hsi_1,\ldots,\hi_k)).\]
There is thus an induced injective map $\BCJ_n(g;\hi_1,\ldots,\hi_k) \rightarrow \BCJ_n(g)$.

\subsection{A potential presentation}
\label{section:bcjpresentation}

Since we will only need to understand $\BCJ_2(g)$, we now restrict to this case.
We start with the map
\[\epsilon\colon \bigoplus_{1 \leq i \leq g} \BCJ_2(g;\his) \rightarrow \BCJ_2(g)\]
induced by the inclusions.  Our goal is to determine the kernel of this map.
Consider some $1 \leq j < k \leq g$.  Let
$f'\colon \BCJ_2(g;\hj,\hks) \rightarrow \BCJ_2(g;\hjs)$ and
$f''\colon \BCJ_2(g;\hj,\hks) \rightarrow \BCJ_2(g;\hks)$ be the
two inclusions.  We have a commutative diagram
\[\begin{tikzcd}[row sep=tiny]
                                                 & \BCJ_2(g;\hjs) \arrow{rd} & \\
\BCJ_2(g;\hj,\hks) \arrow{ru}{f'} \arrow{rd}{f''} &                          & \BCJ_2(g) \\
                                                 & \BCJ_2(g;\hks) \arrow{ru} & 
\end{tikzcd}\]   
Let $d_{jk}\colon \BCJ_2(g;\hj,\hks) \rightarrow \oplus_{i=1}^g \BCJ_2(g;\his)$ be
the composition of the map 
\[(f',f'')\colon \BCJ_2(g;\hj,\hks) \rightarrow \BCJ_2(g;\hjs) \oplus \BCJ_2(g;\hks)\]
with the inclusion of the two-term direct sum into the $g$-term direct sum, and let
\[d = \bigoplus_{1 \leq j < k \leq g} d_{jk} \colon \bigoplus_{1 \leq j<k \leq g} \BCJ_2(g;\hj,\hks)
\rightarrow \bigoplus_{1 \leq i \leq g} \BCJ_2(g;\his).\]
Since we are working over $\bbF_2$, the image of $d$ is contained in the kernel of $\epsilon$, i.e., the composition
\[\bigoplus_{1 \leq j<k \leq g} \BCJ_2(g;\hj,\hks) \stackrel{d}{\longrightarrow} \bigoplus_{1 \leq i \leq g} \BCJ_2(g;\his) 
\stackrel{\epsilon}{\longrightarrow} \BCJ_2(g)\]
is the zero map.  We will prove below that this is actually a presentation of $\BCJ_2(g)$, i.e., the map
$\epsilon$ is surjective and the image of $d$ is the kernel of $\epsilon$.
We remark that this is exactly the kind of presentation that appears in the theory of {\em central stability}
for sequences of representations of the symmetric group.  This was introduced 
by Putman \cite{PutmanCongruence} and further developed by Putman--Sam \cite{PutmanSamNoetherian}, and gives one
way of formalizing Church--Farb's notion of representation stability \cite{ChurchFarbRepStability}.

\subsection{The presentation}

Our result is as follows.  It is the main result in this section.

\begin{lemma}
\label{lemma:bcjpresentation}
Let the notation be as above, and assume that $g \geq 3$.  The following is then exact:
\[\bigoplus_{1 \leq j<k \leq g} \BCJ_2(g;\hj,\hks) \stackrel{d}{\longrightarrow} \bigoplus_{1 \leq i \leq g} \BCJ_2(g;\his) 
\stackrel{\epsilon}{\longrightarrow} \BCJ_2(g) \rightarrow 0.\]
\end{lemma}
\begin{proof}
We must prove that $\epsilon$ is surjective and that $\Image(d) = \ker(\epsilon)$.
Let $H = \HH_1(\Sigma_g^1;\bbF_2)$ and let $S = \{a_1,b_1,\ldots,a_g,b_g\}$ be the standard symplectic basis
for $H$.  Define $\oS = \{\oa_1,\ob_1,\ldots,\oa_g,\ob_g\}$, and set
\[\oS_{\leq 2} = \{1\} \cup \oS \cup \Set{$\os \ot$}{$\os,\ot \in \oS$ distinct}.\]
The set $\oS_{\leq 2}$ is a basis for the $\bbF_2$-vector space $\BCJ_2(g)$.  For $1 \leq i_1 < \cdots < i_k \leq g$, set
\begin{align*}
\oS(\hsi_1,\ldots,\hi_k) &= \oS \setminus \{\oa_{i_1},\ob_{i_1},\ldots,\oa_{i_k},\ob_{i_k}\} \\
\oS_{\leq 2}(\hsi_1,\ldots,\hi_k)   &= \{1\} \cup \oS(\hsi_1,\ldots,\hi_k) \cup \Set{$\os \ot$}{$\os,\ot \in \oS(\hsi_1,\ldots,\hi_k)$ distinct}.
\end{align*}
The set $\oS_{\leq 2}(\hsi_1,\ldots,\hi_k)$ is a basis for the $\bbF_2$-vector space $\BCJ_2(g;\hi_1,\ldots,\hi_k)$.  

Since $g \geq 3$, each element of $\oS_{\leq 2}$ lies in $\oS_{\leq 2}(\hsis)$ for some $1 \leq i \leq g$.  This
implies that $\epsilon$ is surjective.  Moreover, if $\mu \in \oS_{\leq 2}$ lies in
$\oS_{\leq 2}(\hsjs)$ and $\oS_{\leq 2}(\hsks)$ for some distinct $1 \leq j < k \leq g$, then $\mu \in \oS_{\leq 2}(\hsj,\hks)$ and
\[d(\mu) = (\mu,\mu) \in \BCJ_2(g;\hjs) \oplus \BCJ_2(g;\hks) \subset \bigoplus_{1 \leq i \leq g} \BCJ_2(g;\his).\]
Since we are working over $\bbF_2$, we have $\mu = -\mu$.  It follows that quotienting by the image of $d$
identifies $\mu \in \BCJ_2(g;\hjs)$ and $\mu \in \BCJ_2(g;\hks)$.  We conclude that $\Image(d) = \ker(\epsilon)$.
\end{proof}

\section{Proof of Theorem \texorpdfstring{\ref{maintheorem:qkgpres}}{A''}}
\label{section:qkgpresproof}

We close this paper by proving Theorem~\ref{maintheorem:qkgpres}, whose statement
we recall below.\footnote{This implies our main results.  Indeed, in \S \ref{section:theoremaprimeproof} of Part~\ref{part:2} we
proved that Theorem~\ref{maintheorem:qkgpres} implies Theorem~\ref{maintheorem:qkg},
and in \S \ref{section:theoremaproof} -- \S \ref{section:theorembproof} of Part~\ref{part:1} we proved that
Theorem~\ref{maintheorem:qkg} implies Theorems~\ref{maintheorem:genker} and \ref{maintheorem:gencomm}.}  
We remark that the proof of Theorem~\ref{maintheorem:qkgpres} uses a proof
strategy introduced by Minahan--Putman \cite{MinahanPutmanRep}.

\newtheorem*{maintheorem:qkgpres}{Theorem~\ref{maintheorem:qkgpres}}
\begin{maintheorem:qkgpres}
For $g \geq 3$, the lifted BCJ homomorphism $\beta\colon \Pres_g^1 \rightarrow \BCJ_2(g)$ is
an isomorphism.
\end{maintheorem:qkgpres}
\begin{proof}
In fact, we will prove that this holds for $g \geq 2$.
The proof will be by induction on $g$.  The base cases $g=2$ and $g=3$ 
are Lemmas~\ref{lemma:qkgpresg2} and \ref{lemma:qkgpresg3}.  Assume, therefore, that
$g \geq 4$ and that the theorem is true in all genera between $2$ and $g-1$.  We will
prove that $\Pres_g^1$ has a presentation similar to the one we constructed for $\BCJ_2(g)$
in Lemma~\ref{lemma:bcjpresentation}.  Our inductive hypothesis will then allow
us to identify the terms in this presentation with those for $\BCJ_2(g)$.  

This requires some notation.
Let $U$ be an $\bbF_2$-vector space equipped with a symplectic form; for instance,
$U$ might be a symplectic subspace of $H = \HH_1(\Sigma_g^1;\bbF_2)$.  Generalizing our construction of $\Pres_g^1$ from
\S \ref{section:liftedbcj}, define $\Pres(U)$ to be the $\bbF_2$-vector space given by generators and relations as follows:
\begin{itemize}
\item The generators of $\Pres(U)$ are formal symbols $\SymSub{V}$ with $V < U$ a nonzero
symplectic subspace.
\item The relations of $\Pres(U)$ are as follows.  Let $V,W < U$ be nonzero orthogonal symplectic subspaces.
We then have the relation $\SymSub{V} + \SymSub{W} = \SymSub{V \oplus W}$.
\end{itemize}
In Lemma \ref{lemma:liftedbcjcalc}, we gave a formula for calculating the image of a generator
under the lifted BCJ homomorphism $\beta\colon \Pres_g^1 \rightarrow \BCJ_2(g)$.  That same formula
lets us define a lifted BCJ homomorphism
$\Pres(U) \rightarrow \BCJ_2(U)$.  In fact, if $U$ has
genus $h$ then $\Pres(U) \cong \Pres_h^1$ and $\BCJ_2(U) \cong \BCJ_2(h)$, and the lifted
BCJ homomorphism $\Pres(U) \rightarrow \BCJ_2(U)$ can be identified with the usual
lifted BCJ homomorphism $\Pres_h^1 \rightarrow \BCJ_2(h)$.

This construction is functorial in the following sense.  Let $U$ and $U'$ be $\bbF_2$-vector
spaces equipped with symplectic forms and let $\iota\colon U \rightarrow U'$ be
a linear map preserving the symplectic form.  This implies that $\iota$ is injective.
The map $\iota\colon U \rightarrow U'$ induces
a map $\Pres(U) \rightarrow \Pres(U')$.  Unlike for $\BCJ_2(U)$, this induced map
is not obviously injective.

As in \S \ref{section:standardspdecomp},
let $\{a_1,b_1,\ldots,a_g,b_g\}$ be the standard symplectic basis for $H = \HH_1(\Sigma_g^1;\bbF_2)$.
For $1 \leq i \leq g$, let $H(i) = \Span{a_i,b_i}$.  This is a genus-$1$ symplectic subspace of $H$, and
we have an orthogonal decomposition
\[H = H(1) \oplus H(2) \oplus \cdots \oplus H(g).\]
For distinct $1 \leq i_1,\ldots,i_k \leq g$, let $H(\hsi_1,\ldots,\hi_k)$
be the subspace of $H$ obtained by deleting the terms $H(i_j)$ from the
above decomposition, so 
\[H(\hsi_1,\ldots,\hi_k) = \bigoplus_{\substack{1 \leq i \leq g \\ i \neq i_1,\ldots,i_k}} H(i).\]
Finally, define
\[\Pres_g^1(\hsi_1,\ldots,\hi_k) = \Pres(H(\hsi_1,\ldots,\hi_k)).\]
Exactly like in \S \ref{section:bcjpresentation}, the maps induced by the various
inclusion maps fit together into maps
\[\bigoplus_{1 \leq j<k \leq g} \Pres_g^1(\hsj,\hks) \stackrel{D}{\longrightarrow} \bigoplus_{1 \leq i \leq g} \Pres_g^1(\hsis) 
\stackrel{E}{\longrightarrow} \Pres_g^1\]
such that the following diagram commutes:
\[\begin{tikzcd}
\displaystyle{\bigoplus_{1 \leq j<k \leq g} \Pres_g^1(\hsj,\hks)} \arrow{r}{D} \arrow{d}{\beta'} &
\displaystyle{\bigoplus_{1 \leq i \leq g} \Pres_g^1(\hsis)} \arrow{r}{E} \arrow{d}{\beta''} &
\Pres_g^1 \arrow{d}{\beta} 
& \\
\displaystyle{\bigoplus_{1 \leq j<k \leq g} \BCJ_2(g;\hj,\hks)} \arrow{r}{d} & 
\displaystyle{\bigoplus_{1 \leq i \leq g} \BCJ_2(g;\his)} \arrow{r}{\epsilon} &
\BCJ_2(g) \arrow{r}
& 0.
\end{tikzcd}\]
Here the bottom row is the exact sequence from Lemma~\ref{lemma:bcjpresentation}.  The maps $\beta'$
and $\beta''$ are direct sums of the various lifted BCJ homomorphisms.  

For $1 \leq i \leq g$ we have $\Pres_g^1(\hsis) \cong \Pres_{g-1}^1$ and $\BCJ_2(g;\his) \cong \BCJ_2(g-1)$, and
for $1 \leq j < k \leq g$ we have $\Pres_g^1(\hsj,\hks) \cong \Pres_{g-2}^1$ and $\BCJ_2(g;\hj,\hks) \cong \BCJ_2(g-2)$.
Our inductive hypothesis therefore
implies that $\beta'$ and $\beta''$ are isomorphisms.  We remark that this uses the fact that
$g \geq 4$ since this is needed to ensure that $g-2 \geq 2$.  

Since the bottom row of the above diagram is exact and $\beta'$ and $\beta''$ are isomorphisms,
a diagram chase shows that to prove that
$\beta$ is an isomorphism, it is enough to prove that $E$ is surjective:

\begin{unnumberedclaim}
The map $E \colon \bigoplus_{1 \leq i \leq g} \Pres_g^1(\hsis) \rightarrow \Pres_g^1$ is surjective.
\end{unnumberedclaim}

Let $\Lambda$ be the image of $E$.  Our goal is to prove that $\Lambda = \Pres_g^1$.  Note
that $H(1)$ is a genus-$1$ symplectic subspace of $H$ such that $\SymSub{H(1)} \in \Pres_g^1(\hspace{2pt}\widehat{2}\hspace{2pt})$.
It follows that $\SymSub{H(1)} \in \Lambda$.  The group $\Sp_{2g}(\bbF_2)$ acts transitively
on genus-$1$ symplectic subspaces of $H$, so the span of the $\Sp_{2g}(\bbF_2)$-orbit of $\Lambda$
contains all $\SymSub{V}$ with $V$ a genus-$1$ symplectic subspace of $H$.  

We claim that
these generate $\Pres_g^1$.  Indeed, if $W$ is an arbitrary symplectic subspace of $H$ then
there is an orthogonal decomposition $W = V_1 \oplus \cdots \oplus V_h$ with each $V_i$
a genus-$1$ symplectic subspace, so
\[\SymSub{W} = \SymSub{V_1} + \cdots + \SymSub{V_h}.\]
It follows that the $\Sp_{2g}(\bbF_2)$-orbit of $\Lambda$ spans $\Pres_g^1$.  To prove
that $\Lambda = \Pres_g^1$, it is therefore enough to prove that
$\Sp_{2g}(\bbF_2)$ takes $\Lambda$ to $\Lambda$.  For this, it is enough to prove the following:
\begin{itemize}
\item[$(\spadesuit)$] There exist generating sets $S$ for $\Lambda$ and $T$ for $\Sp_{2g}(\bbF_2)$ such that
for $s \in S$ and $\phi \in T$ we have $\phi(s) \in \Lambda$.
\end{itemize}
We will verify this in three steps: first we will construct $S$, then we will construct $T$, and then
finally we will verify $(\spadesuit)$.

\begin{step}{1}
We construct a generating set $S$ for $\Lambda$.
\end{step}

By definition, $\Lambda$ is generated by the images of the maps
$\Pres_g^1(\hsis) \rightarrow \Pres_g^1$ as $i$ ranges over $1 \leq i \leq g$.  Consider some such $1 \leq i \leq g$.  Our inductive
hypothesis applies to $\Pres_g^1(\hsis) \cong \Pres_{g-1}^1$ and says that
$\Pres_{g}^1(\hsis) \cong \BCJ_2(g;\his)$.  Lemma~\ref{lemma:bcjpresentation} implies that
the map
\[\bigoplus_{\substack{1 \leq i' \leq g \\ i' \neq i}} \BCJ_2(g;\hi,\hips) \rightarrow \BCJ_2(g;\his)\]
is surjective.  This map fits into a commutative diagram
\[\begin{tikzcd}
\displaystyle{\bigoplus_{\substack{1 \leq i' \leq g \\ i' \neq i}} \Pres_g^1(\hsi,\hips)} \arrow{r} \arrow{d} & \Pres_g^1(\hsis) \arrow{d}{\cong} \\
\displaystyle{\bigoplus_{\substack{1 \leq i' \leq g \\ i' \neq i}} \BCJ_2(g;\hi,\hips)} \arrow[two heads]{r} & \BCJ_2(g;\his)
\end{tikzcd}\]
Since $g \geq 4$, our inductive hypothesis also applies to all the groups $\Pres_g^1(\hsi,\hips) \cong \Pres_{g-2}^1$ appearing
in this diagram.  Applying our inductive hypothesis, we see that the left-hand vertical arrow is an isomorphism.
We deduce that $\Pres_g^1(\hsis)$ is generated by the images of the maps $\Pres_g^1(\hsi,\hips) \rightarrow \Pres_g^1(\hsis)$
as $i'$ ranges over $1 \leq i' \leq g$ with $i' \neq i$.

Since $\Lambda$ is generated by the images of the maps $\Pres_g^1(\hsis) \rightarrow \Pres_g^1$ as $i$
ranges over $1 \leq i \leq g$, it follows that 
$\Lambda$ is generated by the images of the maps $\Pres_g^1(\hsi,\hips) \rightarrow \Pres_g^1$ as $i$ and $i'$ range over distinct
elements of $\{1,\ldots,g\}$.  For $1 \leq i,i' \leq g$ distinct, let
\[S(i,i') = \Set{$\SymSub{V} \in \Pres_g^1$}{$V$ a symplectic subspace of $H$ satisfying $V \subset H(\hsi,\hips)$}.\]
The set $S(i,i')$ generates the image of the map $\Pres_g^1(\hsi,\hips) \rightarrow \Pres_g^1$.  We conclude
that $\Lambda$ is generated by the set
\[S = \bigcup_{\substack{1 \leq i,i' \leq g \\ i \neq i'}} S(i,i').\]

\begin{step}{2}
We construct a generating set $T$ for $\Sp_{2g}(\bbF_2)$.
\end{step}

For a symplectic subspace $V$ of $H$, we can identify $\Sp(V)$ with the subgroup
of $\Sp_{2g}(\bbF_2)$ consisting of $\phi \in \Sp_{2g}(\bbF_2)$ that act trivially
on $V^{\perp}$.  Let
\[T = \bigcup_{\substack{1 \leq j,j' \leq g \\ j \neq j'}} \Sp(H(j) \oplus H(j')) \subset \Sp_{2g}(\bbF_2).\]
This is a generating set for $\Sp_{2g}(\bbF_2)$; for instance, it contains
all the elementary symplectic matrices (see, e.g., \cite{HahnOmeara}).

\begin{step}{3}
We prove $(\spadesuit)$: for $s \in S$ and $\phi \in T$ we have $\phi(s) \in \Lambda$.
\end{step}

Consider $s \in S$ and $\phi \in T$.  By definition, we can find:
\begin{itemize}
\item $1 \leq i,i' \leq g$ with $i \neq i'$ such that $s=\SymSub{V}$ with $V$ a symplectic subspace of $H$ satisfying $V \subset H(\hsi,\hips)$; and
\item $1 \leq j,j' \leq g$ with $j \neq j'$ such that $\phi \in \Sp(H(j) \oplus H(j'))$.
\end{itemize}
To show that $\phi(s) \in \Lambda$, we must prove that $\SymSub{\phi(V)} \in \Lambda$. 

There are two cases.  The first is that $\{j,j'\} = \{i,i'\}$.  Since
\[H = H(\hsi,\hips) \oplus H(i) \oplus H(i') = H(\hsi,\hips) \oplus H(j) \oplus H(j'),\]
the element $\phi \in \Sp(H(j) \oplus H(j'))$ acts trivially on $H(\hsi,\hips)$.  It follows
that $\phi(V) = V$ and there is nothing to prove.

The second case is that $\{j,j'\} \neq \{i,i'\}$.  Swapping $i$ and $i'$ if necessary, we
can assume that $i \notin \{j,j'\}$.  This implies that $\phi \in \Sp(H(j) \oplus H(j'))$ acts
trivially on $H(i)$.  Since we have an orthogonal decomposition
$H = H(\hsis) \oplus H(i)$
it follows that $\phi$ takes $H(\hsis)$ to $H(\hsis)$.  Since $V \subset H(\hsi,\hips) \subset H(\hsis)$, we
conclude that $\phi(V) \subset H(\hsis)$ and thus that $\SymSub{\phi(V)}$ lies in
the image of the map $\Pres_g^1(\hsis) \rightarrow \Pres_g^1$.  This implies that
$\SymSub{\phi(V)} \in \Lambda$, as desired.
\end{proof}

\end{document}